\documentclass[reqno,11pt]{amsart}

\usepackage{graphicx}
\usepackage{relsize}
\usepackage[ansinew]{inputenc}
\usepackage{amsfonts,epsfig}
\usepackage{latexsym}
\usepackage{mathabx}
\usepackage{amsmath}
\usepackage{amssymb}
\usepackage{color}
\usepackage{hyperref}

\newtheorem{theorem}{Theorem}
\newtheorem{lemma}[theorem]{Lemma}
\newtheorem{corollary}[theorem]{Corollary}

\newtheorem{proposition}[theorem]{Proposition}

\theoremstyle{definition}

\theoremstyle{remark}

\numberwithin{equation}{section}

\numberwithin{equation}{section}

\newcommand{\D}{\mathbb{D}}

\renewcommand{\phi}{\varphi}

\newcommand{\T}{\mathbb{T}}

\begin{document}
\title[Localization and weighted theory on the Bergman spaces]{Localization and weighted theory on the Bergman spaces}

	
\thanks{The research is partially supported by National Natural Science Foundation of China
(Grant Numbers No. 12571136 and No.12471116) and 2025CDJ-IAIS YB-004 (Chongqing University).}

\makeatletter
\@namedef{subjclassname@2020}{\textup{2020} Mathematics Subject Classification}
\makeatother

\author[Yongjiang Duan]{Yongjiang Duan}
\address[Yongjiang Duan]{Department of Mathematics\\
		Jinan University\\ Guangzhou\\ Guangdong 510632\\ P.R.China}
	\email{yjduan@jnu.edu.cn}
	
	\author[Junhan Hong]{Junhan Hong}
	\address[Junhan Hong]{School of Mathematics and Statistics\\
		Northeast Normal University\\Changchun\\Jilin 130024\\ P.R.China}
	\email{hongjh815@nenu.edu.cn}

\author[Siyu Wang]{Siyu Wang}
\address[Siyu Wang]{Department of Physics and Mathematics\\
		 University of Eastern Finland\\ P.O.Box 111\\ FI-80101 Joensuu\\ Finland}
	\email{siyu.wang@uef.fi; wangsy696@nenu.edu.cn}

\author[Zipeng Wang]{Zipeng Wang}
\address [Zipeng Wang]{College of Mathematics and Statistics, Chongqing University, Chongqing,
401331, P.R.China}
\email{zipengwang@cqu.edu.cn}
	
\begin{abstract}
In this paper, we introduce the class of weakly localized operators on the weighted Bergman spaces induced by two-sided doubling weights. We first establish the characterization of the compactness of the operator belonging to the norm closure of weakly localized operators in terms of the boundary vanishing of the Berezin transform. Our approach here is different from the methods in the literature since the  ``translation" operators are not served us a direct applicable tool in our setting. We further characterize the compactness of Toeplitz operators with bounded symbols in the two-weight setting, a boundedness criterion of a two-weight Bergman projection is also presented.
\end{abstract}

\maketitle

\section{Introduction and main results}

The compactness of Toeplitz operators on Bergman spaces has been a central topic in operator theory for several decades. A fundamental result for the unweighted Bergman space asserts that a bounded Toeplitz operator is compact if and only if its Berezin transform vanishes at the boundary of the unit disk \cite{AZ}. This criterion, first established through the Berezin transform, and later extended via the theory of weakly localized operators, has served as a prototype for studying compactness characterization in more general function spaces, for example, Fock spaces and weighted Bergman spaces, see \cite{Cj,Englis,IMW,SW,XZ} for more details.

Weights in harmonic analysis naturally originated in the study of Hardy-Littlewood maximal operators on weighted Lebesgue spaces and are closely related to a large class of singular integral operators, including the Hilbert transform and Riesz transforms, and hence play a fundamental role in modern harmonic analysis. They also provide important connections among harmonic analysis, partial differential equations, geometry, potential theory, probability and other branches of mathematics. Recently, weights have become increasingly important in operator theory.
A notable example is the Sarason conjecture on the Bergman space. Aleman, Pott and Reguera \cite{APR1} showed that the boundedness of the associated Toeplitz products is equivalent to a two-weight boundedness problem for the Bergman projection with a special pair of weights. They also constructed a counterexample to the conjecture using a dyadic model approach. The two-weight boundedness problem for the Bergman projection with arbitrary weights, however, remains open in general.
This question have explicit solution in certain radial settings, see \cite{PR4,PR2} for example.

In this paper, we consider weights $\omega$ on the unit disk $\mathbb{D}$ which are nonnegative functions and integrable with respect to the Lebesgue measure. A weight $\omega$ is radial if $\omega(z)
=\omega(|z|)$ for all $z\in\mathbb{D}$. Several classes of radial weights on the unit disk have been introduced and investigated in the context of analytic function spaces and related operator theory. The radial doubling weights, denoted by $\widehat{\mathcal{D}}$, constitute an important class for which many desirable properties of the classical Bergman theory remain valid, see for instance \cite{LW-,Pe,PR5, PR4,PR2,PRWW} and the references therein. Recall that a radial weight $\omega \in \widehat{\mathcal{D}}$ if $\widehat{\omega}(z)=\int_{|z|}^{1}\omega(s)ds$ satisfies the following condition, i.e., there exists $C=C(\omega)\geq1$ such that
$$
\widehat{\omega}(r)\leq C\widehat{\omega}\left(\frac{1+r}{2}\right),\quad{0\leqslant r<1}.
$$
If there exist $K=K(\omega)>1$ and $C=C(\omega)>1$ such that
\begin{equation}\label{Ddef}
\widehat{\omega}(r)\geq C\widehat{\omega}\left(1-\frac{1-r}{K}\right),\quad{0\leqslant r<1},
\end{equation} then we say $\omega\in \widecheck {\mathcal{D}}$.
Moreover, let $\mathcal{D}=\widehat{\mathcal{D}}\cap \widecheck{\mathcal{D}}$. Throughout the paper, we assume $\widehat{\omega}(r)>0$ for all $0\leq r<1$.

For $0<p<\infty$ and a weight $\omega$ on $\mathbb{D}$, the space $L_{\omega}^{p}$ consists of functions $f$ on $\mathbb{D}$ such that
   $$
   \|f\|_{L_{\omega}^{p}}=\left(\int_{\D}|f(z)|^{p}\omega(z)dA(z)\right)^{\frac{1}{p}}<\infty,
   $$
where $dA$ is the normalized area measure on $\mathbb{D}$. The weighted Bergman space induced by the weight $\omega$ is defined by $A_\omega^p= L_\omega^p\cap H(\mathbb{D})$, where $ H(\mathbb{D})$ denotes the space of all analytic functions on $\mathbb{D}$.
As usual, we write $A^p_{\alpha}$ for the standard weighted Bergman space when $\omega(z)=(\alpha+1)(1-|z|^2)^\alpha$ with $-1<\alpha<\infty$. When $\omega$ is a radial weight, $A_\omega ^2$ is a reproducing kernel Hilbert space whose reproducing kernels are given by
$$B_z^{\omega}(\zeta)=\sum\limits_{m=0}^{\infty}\frac{(\bar{z}\zeta)^m}{2\omega_{2m+1}},$$ where $$\omega_{x}=\int_{0}^{1} t^{x} \omega(t)dt,\quad 0< x<\infty.$$ The normalized reproducing kernel is written as $b_{p,z}^\omega =B_z^\omega/\|B_z^\omega\|_{A_\omega^p}$. We simply write $b_z^\omega$ for $p=2$.

Recall that for a finite Borel measure $\mu$, the Toeplitz operator $T^\omega_{\mu}$ induced by $\mu$ is densely defined by
$$
T^\omega_{\mu}(g)(z)=\int_{\mathbb{D}} g(\zeta)\overline{B_z^\omega(\zeta)}d\mu(\zeta), \quad g\in H^\infty,
$$
where $H^\infty$ is the space of all bounded analytic functions on $\D$.
When $d\mu=f\omega dA$ for $f\in L_\omega^1$, we write $T^\omega_{\mu}=T^\omega_f$, which is densely defined as
$$
T_f^\omega(g)(z)=\int_{\mathbb{D}}f(\zeta)g(\zeta)\overline{B_z^\omega(\zeta)}\omega(\zeta)dA(\zeta), \quad g\in H^\infty.
$$
The Berezin transform of a bounded linear operator $T:A_\omega^2\rightarrow A_\omega^2$ is
$$
\widetilde{T}(z)=\langle T(b_z^\omega),b_z^\omega \rangle_{A_\omega^2}.
$$
The Berezin transform is a fundamental and important tool in the study of the Toeplitz operators on the classical Bergman space, see \cite{Zhu2} for example.

In this paper, we concentrate on the class of operators that are ``localized".
In \cite{AZ}, Axler and Zheng actually employed the idea of operator localization, though not explicitly named as such at the time, showing that for any operator $T$ that is a finite sum of finite products of Toeplitz operators with $L^{\infty}$-symbols, compactness of $T$ on the classical Bergman space $A^2$ is equivalent to the vanishing of its Berezin transform $\widetilde{T}$ at the boundary of the unit disk. Very recently, Li and Wang \cite{LW} established this result for weighted Bergman space induced by $\widehat{\mathcal{D}}$ weights. In \cite{Su}, Su\'{a}rez showed that a bounded operator $T$ on the Bergman space is compact if and only if $T$ belongs to the Toeplitz algebra $\mathcal{T}$ (the closed algebra generated by Toeplitz operators with bounded symbols) and $\widetilde{T}(z) \rightarrow 0$ as $|z| \rightarrow 1$, where the proof again relied on localization techniques. This result was subsequently extended by Mitkovski, Su\'{a}rez and Wick in \cite{MSW} to the setting of standard weighted Bergman spaces on the unit ball in $\mathbb{C}^n$.

Localization was first discovered in \cite{XZ} as a powerful tool for analyzing operators on reproducing kernel Hilbert spaces. Isralowitz, Mitkovski and Wick \cite{IMW} further explored this idea by introducing the notion of weakly localized operators on the Bergman space. Moreover, they demonstrated that the norm closure of the weakly localized operators contains the Toeplitz algebra and that, within this closure, compactness is equivalent to the vanishing of its Berezin transform at the boundary. Later in the same year, Xia \cite{Xia} proved that the Toeplitz algebra coincides with the norm closure of the $*$-algebra of weakly localized operators, presenting the importance of weakly localized operators in relating compactness to the boundary vanishing of the Berezin transform. Recently, Dewage \cite{Dewage} generalized Xia's result to bounded symmetric domains by using representation theory and quantum harmonic analysis technology.

In this paper, we introduce weakly localized operators on weighted Bergman spaces induced by $\mathcal{D}$ weights and investigate the corresponding compactness phenomena. We prove that the Toeplitz algebra is contained in the norm closure of the weakly localized operators on weighted Bergman spaces induced by $\mathcal{D}$ weights, retaining the feature discovered in \cite{IMW}. However, this may fail when it comes to the class $\widehat{\mathcal{D}}$, for example, let
 $\omega(r)=(1-r)^{-1} (\log\frac{e}{1-r})^{-2}$ for $0\leq r<1$, a weight which belongs to $\widehat{\mathcal{D}}\setminus \mathcal{D}$. Indeed, every operator $T$ belonging to the norm of closure of weakly localized operators is bounded on $A_{\omega}^{p}$ and satisfies $Tb_0^\omega=0$, whereas the identity operator $I=T_1^\omega$ belongs to the Toeplitz algebra but does not satisfy this property.

In order to introduce the weakly localized operators on weighted Bergman spaces induced by~$\mathcal{D}$ weights,
we first recall one useful characterization for weights in $\widecheck{\mathcal{D}}$, presented in \cite[Lemma B]{PRS}, showing that for a radial weight $\omega$, then $\omega \in \widecheck{\mathcal{D}}$ if and only if there exist $C=C(\omega)>0$ and $\alpha_0(\omega)>0$ such
that
	\begin{equation}\label{Dcheck}
	\widehat{\omega}(t) \leqslant C\left(\frac{1-t}{1-r}\right)^{\alpha_0(\omega)} \widehat{\omega}(r), \quad 0 \leqslant r \leqslant t<1.
	\end{equation}
Since this inequality remains valid for every smaller positive exponent, throughout this paper we fix $0<\alpha(\omega)<\min\{1,\alpha_0(\omega)  \}$ and use $\alpha(\omega)$ in place of $\alpha_0(\omega)$ in \eqref{Dcheck}.
Denote $\gamma(\omega)=\frac{1-\alpha(\omega)}{1+\alpha(\omega)}.$
Let $1<p<\infty$, $p'$ denote the conjugate index of $p$ and $\omega\in\mathcal{D}$.
A linear operator $T$, defined on the linear span of the reproducing kernels, is said to be weakly localized for $A_\omega^p$, denoted by $T\in \mathrm{WL}_\omega^p$, if there exists
 \begin{equation}\label{delta}
 0<\delta<\left \{
\begin{aligned}
&p\big(1-\gamma(\omega)\big), &&  1<p\leq 2, \\
&p'\big(1-\gamma(\omega)\big), &&  p>2,
\end{aligned}
\right.
\end{equation} such that the following conditions hold,
\begin{equation}\label{WL1}
\sup_{z\in\mathbb{D}}\int_{\mathbb{D}}|\langle Tb_{z}^{\omega}, b_{\zeta}^{\omega}\rangle_{A_\omega^{2}}|
\frac{{\|B_{z}^{\omega}\|}^{1-\frac{\delta}{p'}}_{A_\omega^2}}{{\|B_{\zeta}^{\omega}\|}^{1-\frac{\delta}{p'}}_{A_\omega^2}}
d\lambda_\omega(\zeta)< \infty,
\end{equation}
\begin{equation}\label{WL2}
\sup_{z\in\mathbb{D}}\int_{\mathbb{D}}|\langle T^{*}b_{z}^{\omega}, b_{\zeta}^{\omega}\rangle_{{A_\omega^2}}|
\frac{{\|B_{z}^{\omega}\|}^{1-\frac{\delta}{p}}_{A_\omega^2}}{{\|B_{\zeta}^{\omega}\|}^{1-\frac{\delta}{p}}_{A_\omega^2}}
d\lambda_\omega(\zeta)< \infty;
\end{equation}
and
\begin{equation}\label{WL3}
\lim_{r\rightarrow \infty}\sup_{z\in\mathbb{D}}\int_{\mathbb{D}\setminus D(z,r)}|\langle Tb_{z}^{\omega}, b_{\zeta}^{\omega}\rangle_{A_\omega^{2}}|
\frac{{\|B_{z}^{\omega}\|}^{1-\frac{\delta}{p'}}_{A_\omega^2}}{{\|B_{\zeta}^{\omega}\|}^{1-\frac{\delta}{p'}}_{A_\omega^2}}
d\lambda_\omega(\zeta)=0,
\end{equation}
$$\label{WL4}
\lim_{r\rightarrow \infty}\sup_{z\in\mathbb{D}}\int_{\mathbb{D}\setminus D(z,r)}|\langle T^{*}b_{z}^{\omega}, b_{\zeta}^{\omega}\rangle_{A_\omega^{2}}|
\frac{{\|B_{z}^{\omega}\|}^{1-\frac{\delta}{p}}_{A_\omega^2}}{{\|B_{\zeta}^{\omega}\|}^{1-\frac{\delta}{p}}_{A_\omega^2}}
d\lambda_\omega(\zeta)=0,
$$
where
$$
d\lambda_\omega(\zeta)=\frac{\omega(\zeta)}{\widehat{\omega}(\zeta)(1-|\zeta|)}dA(\zeta),\quad \zeta\in\mathbb{D},
$$
and $D(z,r)=\{{w\in \mathbb{D}:\beta(w,z)<r}\}$. Here $$\beta(z,w)=\frac{1}{2}\log \frac{1+|\varphi_z(w)|}{1-|\varphi_z(w)|},\quad z,w\in\mathbb D,$$
denotes the Bergman metric in $\mathbb D$, where $\varphi_z(w)=\frac{z-w}{1-\bar{z}w}$.

Now we state our main result on localization, characterizing
the compact operators, within the norm closure of the weakly localized operators on weighted Bergman spaces induced by $\mathcal{D}$ weights, in terms of the vanishing property of the Berezin transform at the boundary of the unit disk.

\begin{theorem}\label{main}
Let $\omega\in \mathcal{D}$ and $1<p<\infty$. If $T$ belongs to the norm closure of $\mathrm{WL}_\omega^p$, then $T$ is compact on $A_\omega^p$ if and only if $\lim\limits_{|z|\to 1^-}\widetilde{T}(z)=0.$
\end{theorem}

Although the proof of \cite[Theorem 1.5]{IMW} provides a useful guideline for establishing Theorem~\ref{main}, several essential differences arise in our setting. In \cite{MW}, Mitkovski and Wick studied the reproducing kernel thesis (RKT) for boundedness and compactness of various operators on Bergman-type spaces. This class, however, does not include all $\mathcal{D}$-weighted Bergman spaces~$A_\omega^2$, since the A.5 condition introduced in \cite{MW}
\begin{equation} \label{A5}
|\langle b_z^\omega, b_\zeta^\omega   \rangle_{A_\omega^2}|\asymp \frac{1}{\|B^\omega_{\phi_{z}(\zeta)}\|_{A_\omega^2}},\quad{z,\zeta\in\mathbb{D}},
\end{equation}
is not, in general, satisfied by arbitrary weights in $\mathcal{D}$. For example, the typical $\mathcal{D}$ weight $\omega(z)=(1-|z|)\big(\log \frac{e}{1-|z|}\big)^{t}$ with $t\in  (0,1)$ fails to satisfy \eqref{A5}. As was pointed out in \cite{IMW}, it is of interest to establish the corresponding result in the absence of the A.5 condition. Here, we achieve this through a refined analysis of weights in $\mathcal{D}$, relying in particular on a precise integral estimate for a modified reproducing kernel obtained in \cite{PRWW}. Another essential difficulty is that the ``translation" operators available in the classical setting are no longer directly applicable. This issue was also noted in \cite{IMW} as a natural circumstance of interest. For a bounded operator $T$ on $A_{\omega}^{p}$, a crucial step in the proof of Theorem \ref{main} is to show that the boundary vanishing of its Berezin transform $\widetilde{T}$ is equivalent to
$$
 \limsup_{|z|\rightarrow 1^-} \left( \sup_{\zeta\in D(z,r)} \left|\langle Tb_{p,z}^\omega,b_{p',\zeta}^\omega \rangle_{A_\omega^2}\right|\right)=0.$$
Different from the typical translation operators used in the classical Bergman setting, we instead employ an affine transform taking advantage of the fact that the hyperbolic geometry of Bergman balls is preserved.

Building on our result concerning weakly localized operators, the second aim of this paper is to investigate the compactness of Toeplitz operators with bounded symbols in the two-weight setting. Since every Toeplitz operator with bounded symbol is weakly localized, the one-weight compactness characterization via the boundary vanishing of the Berezin transform follows as an immediate application of Theorem \ref{main}. We then show that a corresponding two-weight case is indeed equivalent to the one-weight setting.

In \cite{SW}, Stockdale and Wagner showed that, for $p\in(1,\infty)$ and $u\in L^\infty$, if $\sigma$ belongs to
the intersection of $B_p$ and the reverse H\"{o}lder class, then the Toeplitz operator $T_u$ is compact on~$A_\sigma^p$ if and only if its Berezin transform $\widetilde{T_u}(z)\rightarrow 0$ as $|z|\rightarrow 1^-$. Recently, Chen \cite{Cj} further proved that this result is valid for arbitrary $\sigma\in B_{p}$. One can refer to \cite{DGHY,DGWW2,PRS,TV1,YZ,Zhu2} for more characterizations on the boundedness and compactness of Toeplitz operators on weighted Bergman spaces.
Now we give the compactness characterization in the two-weight setting.
\begin{theorem}\label{main2}
Let $\nu\in \mathcal{D}$ and $1<p<\infty$. For $f\in L^\infty$, and for every weight $\omega$ satisfying $\omega\nu^{-1}\in B_p(\nu)$, then the following statements are equivalent:
\begin{itemize}
\item [(i)] $T_f^\nu$ is compact on $A_\omega^p$;
\item[(ii)]$T_f^\nu$ is compact on $A_\nu^p$;
\item[(iii)]$
\lim\limits_{|z|\to 1^-}\widetilde{T}(z)=0 .
$
\end{itemize}
\end{theorem}

To establish Theorem~\ref{main2}, several further challenges that differ substantially from those in the classical approach arise. Our path to build up an adaptable compact extrapolation theorem in out setting is based on a complete characterization of the boundedness of the two-weight Bergman projection in terms of the class of $B_p(\nu)$ weights. This method enables us to obtain Theorem~\ref{main2} for a class of $B_p(\nu)$ weights satisfying certain reverse H\"older condition. However, the reverse H\"older condition on $\omega\nu^{-1}$ and $(\omega\nu^{-1})^{-\frac{p'}{p}}$ may fail for some $\nu\in\mathcal{D}$ and non-radial two-sided doubling weights $\omega$. To get rid of this restriction, we first find an appropriate auxiliary weight
$$\sigma_{\omega,\nu ,r}(z)=\frac{\omega\big(D(z,r)\big)}{\nu \big(D(z,r)\big)}\nu(z),\quad z\in\mathbb{D},$$
such that
$$
\|f\|_{A^p_{\sigma_{\omega,\nu ,r}}}\asymp\|f\|_{A^p_\omega}, \quad f\in H(\D),
$$
for each fixed $r=r(\omega,\nu)$ sufficiently large. This estimate transforms the $A^p_\omega$-norm of any analytic function into its $A^p_{\sigma_{\omega,\nu ,r}}$-norm, where the values of $\sigma_{\omega,\nu ,r}$ are comparable within each hyperbolic disk. This serves us as a useful tool here, and may also have applications in the future work.
Our further argument is based on a dyadic tree together with a stopping-time procedure, adapting the method used in \cite{APR2}. The dyadic structure provides a natural substitute for decomposing the weight $\sigma_{\omega,\nu ,r}$, which allows us to iterate along the tree.

Apart from establishing the compact extrapolation theorem, we are also led to provide a characterization of the boundedness of the two-weight Bergman projection, which guarantees that the Toeplitz operators with bounded symbols are bounded and in turn allows us to proceed with the study of the compactness criteria. 

Recall that the Bergman projection
$$
P_\omega(f)(z)=\int_{\mathbb{D}}f(\zeta)\overline{B_z^{\omega}(\zeta)}\omega(\zeta)dA(\zeta)
$$
is an orthogonal projection from $L^2_{\omega}$ onto $A^2_{\omega}$, and the maximal Bergman projection is
$$
P_\omega^+(f)(z)=\int_{\mathbb{D}}f(\zeta)|B_z^{\omega}(\zeta)|\omega(\zeta)dA(\zeta).
$$
Our approach to the two-weight Bergman projection relies on a weighted extension of the characterization of the classical $B_{p}$ weights.
For $1<p<\infty$ and a weight $\nu$, we say that~$\omega\in B_p(\nu)$ if $\omega\nu$ is integrable and
\begin{equation}\label{Bp}
[\omega]_{B_p(\nu)}=\sup_S \frac{\int_{S}\omega\nu dA}{\nu(S)} \bigg(\frac{\int_{S}\omega^{-\frac{p'}{p}}\nu dA}{\nu(S)} \bigg)^{p-1} <\infty,
\end{equation}
where the supremum is taken over all Carleson squares $S$ (see \cite{PRW}). We write $\omega(E)=\int_{E}\omega dA$ if $E$ is a Lebesgue measurable subset of $\mathbb{D}$ and $\omega$ is a weight on $\mathbb{D}$.

Now we present our result for the two-weight Bergman projection.
\begin{theorem}\label{main1}
Let $1<p<\infty$, $\nu\in \mathcal{D}$, and let $\omega\in L_{\omega,\mathrm{loc}}^1$ be positive. Then the following statements are equivalent:
\begin{itemize}
\item[(i)] $P_\nu^+:L^p_{\omega\nu}\to L^p_{\omega\nu}$ is bounded;
\item[(ii)] $P_\nu:L^p_{\omega\nu}\to L^p_{\omega\nu}$ is bounded;
\item[(iii)] $P_\nu:L^p_{\omega\nu}\to L^{p,\infty}_{\omega\nu}$ is bounded;
\item[(iv)] $\omega\in B_p(\nu)$.
\end{itemize}
Moreover,
$$
\|P_\nu^+\|_{L_{\omega\nu}^p\to L_{\omega\nu}^p}\lesssim [\omega]_{B_p(\nu)}^{\max \{1,\frac{1}{p-1}\}}.
$$
\end{theorem}
The study of the two-weight Bergman projection goes back to the foundational works~\cite{Be,BB}, while the sharp norm estimate was established later in \cite{PottRe}. 
Pel\'{a}ez, R\"{a}tty\"{a} and Wick \cite[Theorem 2]{PRW} subsequently extended these results to weights in~$\mathcal{D}$, for which the associated reproducing kernels $B_{z}^{\nu}$ admit certain integral representation. Following this line of research, we obtain the corresponding characterizations without relying on the kernel structure. A key ingredient in showing the boundedness of the maximal Bergman projection is an optimal off-diagonal pointwise upper estimate for the Bergman reproducing kernel, recently established in \cite[Theorem 1]{PRWW}. More precisely, for $\nu\in\widehat{\mathcal{D}}$,
\begin{equation}\label{kernel-}
\left|B^\nu_{a}(z)\right|
	\lesssim\frac{1}{\nu_{\frac2{|1-\overline{a}z|}}|1-\overline{a}z|},\quad a,z\in\mathbb{D}.
\end{equation}
The necessity of the $B_p(\nu)$ condition follows by combining a local lower pointwise estimate of the Bergman kernel, presented in \cite[Lemma 7]{PRS}, with an appropriately geometric construction.

This paper is organized as follows. In Section 2, we show that the Toeplitz algebra is a subalgebra of the norm closure of $\mathrm{WL}_\omega^p$. Then we prove Theorem \ref{main}, which characterizes the compact operators in $\mathrm{WL}_\omega^p$ in terms of the vanishing property. As an application, we study when Toeplitz operators with symbols in $\mathrm{BMO}_\omega^p$ belong to $\mathrm{WL}_\omega^p$. In Section 3, we first prove Theorem~\ref{main1}, which gives the boundedness of Bergman projection $P_\nu:L_\omega^p\to L_\omega^p$ and consequently guarantees the boundedness of Toeplitz operators with bounded symbols. Then we prove Theorem \ref{main2} based on the compactness argument.

In this paper, we denote by $a\lesssim b$ if there exists a positive constant $C=C(\cdot)$ such that $a\leq Cb$, here the constant $C(\cdot)$ depends on the parameters indicated in the parenthesis, varying under different circumstances. And we also denote by $a\asymp b$ if both $b\lesssim a$ and $b\gtrsim a$ hold.

\section{The operator class $\mathrm{WL}_\omega^p$}
\subsection{Weakly localization and Toeplitz algebra}

In this subsection, we first show that the identity operator is weakly localized, in preparation for later proving that Toeplitz algebra is contained in the norm closure of $\mathrm{WL}_\omega^p$.

To begin with, we recall two useful characterizations of weights in $\widehat{\mathcal{D}}$. For a radial weight~$\omega$, \cite[Lemma 2.1]{Pe} shows that $\omega\in \widehat{\mathcal{D}}$ if and only if there exist constants $C=C(\omega)\geq1$ and $\beta=\beta(\omega)>0$ such that
		\begin{equation}\label{Dhat1}
        \widehat{\omega}(r)\leq C\left(\frac{1-r}{1-t}\right)^\beta\widehat{\omega}(t),
		\quad{0\leq r\leq t<1},
        \end{equation}
if and only if
        \begin{equation}\label{Dhat2}
        \omega_x\asymp \widehat{\omega}\left(1-\frac{1}{x}\right),\quad 1\leq x<\infty.
        \end{equation}
We next recall several integral estimates for Bergman kernels that will play a crucial role later. Let $\omega\in \widehat{\mathcal{D}}$, $\nu$ be a radial weight, $0<p<\infty$ and $n \in \mathbb{N} \cup\{0\}$. If, in addition, $\nu\in\widehat{\mathcal{D}}$, then \cite[Theorem 1]{PR4}\label{Ne1} gives
		\begin{equation}\label{kernel-1}
\left\|(B_{z}^{\omega}\right)^{(n)}\|_{A_{\nu}^{p}}^{p}\asymp\int_{0}^{|z|}
\frac{\widehat{\nu}(t)}{\widehat{\omega}(t)^{p}(1-t)^{p(n+1)}}dt,
		\quad{|z| \rightarrow 1^{-}}.
        \end{equation}
In particular, for $1<p<\infty$,
		\begin{equation}\label{kernel-2}
      \left\|B_{z}^{\omega}\right \|_{A_{\omega}^{p}}^{p}\asymp
		\frac{1}{\widehat{\omega}(z)^{p-1}(1-|z|)^{p-1}}, \quad{z \in \mathbb{D}}.
        \end{equation}
A corresponding estimate for a modified Bergman kernel was recently established in \cite[Corollary 2]{PRWW} as a consequence of \eqref{kernel-}. To be precise, 
\begin{equation}\label{kernel-3}
\int_{\mathbb{D}}\left|(1-\bar{z} \zeta)^k (B_z^\omega)^{(n)}(\zeta)\right|^p \nu(\zeta) d A(\zeta)
\lesssim  |z|^{np}\left(\int_0^{|z|} \frac{\widehat{\nu}(t)}{\widehat{\omega}(t)^p (1-t)^{p(n+1-k)}} d t+1\right), ~z \in \mathbb{D},
\end{equation}
 for each fixed $k \in \mathbb{R}$.
An earlier result in the case $2\le p<\infty$, $k\in\mathbb{N}$ and $n=0$ was obtained in \cite[Lemma 5]{PR3}.

The following two lemmas show that the identity operator is weakly localized.
\begin{lemma}\label{kb}
Let $\omega\in \mathcal{D}$. Then there exists a constant $C=C(\omega)>0$ such that
$$
\sup_{z\in\mathbb{D}}\int_{\mathbb{D}}\left|\langle b_z^\omega, b_{\zeta}^\omega\rangle_{A_\omega^2}\right|\frac{\|B_z^\omega\|_{A_\omega^2}^{\kappa}}{\|B_\zeta^\omega\|_{A_\omega^2}^{\kappa}} d\lambda_\omega(\zeta)\leq C,
$$
for each fixed $\gamma(\omega)<\kappa<1.$
\end{lemma}
\begin{proof}
Let $\eta(s)=\widehat{\omega}(s)^{\frac{1+\kappa}{2}}(1-|s|)^{\frac{1+\kappa}{2}-2}$ for all $0\leq s<1$.
For $\gamma(\omega)=\frac{1-\alpha(\omega)}{1+\alpha(\omega)}<\kappa<1$, \eqref{Dcheck} and \eqref{Dhat1} show that $\eta\in \mathcal{D}$ and $\widehat{\eta}(s)\asymp \widehat{\omega}(s)^{\frac{1+\kappa}{2}}(1-|s|)^{\frac{1+\kappa}{2}-1}$.
It follows from \cite[Lemma 3]{PPR} that $$
\int_r^1 \frac{\omega(t) \widehat{\eta}(t)}{\widehat{\omega}(t)} d t \lesssim \widehat{\eta}(r), \quad 0 \leqslant r<1.
$$
This together with \eqref{kernel-2} and \eqref{kernel-1} shows that
$$
\begin{aligned}
\int_{\mathbb{D}}\left|\langle b_z^\omega, b_{\zeta}^\omega\rangle_{A_\omega^2}\right|\frac{\|B_z^\omega\|_{A_\omega^2}^{\kappa}}{\|B_\zeta^\omega\|_{A_\omega^2}^{\kappa}} d\lambda_\omega(\zeta)
\asymp& \big(\widehat{\omega}(z)(1-|z|)\big)^{\frac{1-\kappa}{2}}\int_{\mathbb{D}} |B_z^\omega(\zeta)|\frac{\omega(\zeta)\widehat{\eta}(\zeta)}{\widehat{\omega}(\zeta)}dA(\zeta)\nonumber\\
\asymp &\big(\widehat{\omega}(z)(1-|z|)\big)^{\frac{1-\kappa}{2}}\int_{0}^{|z|} \left(\int_{r}^1\frac{\omega(s)\widehat{\eta}(s)}{\widehat{\omega}(s)}ds\right) \frac{1}{\widehat{\omega}(r)(1-r)} dr\nonumber\\
\lesssim &\big(\widehat{\omega}(z)(1-|z|)\big)^{\frac{1-\kappa}{2}}\int_{0}^{|z|} \frac{dr}{\widehat{\omega}(r)^{\frac{1-\kappa}{2}}(1-r)^{\frac{3-\kappa}{2}}}
\lesssim  1, \quad{z\in\mathbb{D}},
\end{aligned}
$$
which completes the proof.
\end{proof}
For a weight $\omega$ and $s\in\mathbb{R}$, we denote $\omega_{[s]}(z)=\omega(z)(1-|z|)^{s}$ for all $z\in\mathbb{D}$.
\begin{lemma}\label{kc}
Let $\omega\in \mathcal{D}$. Then
$$
\lim_{R\rightarrow \infty}\sup_{z\in\mathbb{D}}
\int_{\mathbb{D}\setminus D(z,R)}\left|\langle b_z^\omega, b_{\zeta}^\omega\rangle_{A_\omega^2}\right|\frac{\|B_z^\omega\|_{A_\omega^2}^{\kappa}}{\|B_\zeta^\omega\|_{A_\omega^2}^{\kappa}} d\lambda_\omega(\zeta)=0,
$$
for each fixed $\gamma(\omega)<\kappa<1$.
\end{lemma}
\begin{proof}
Note that $\frac{2}{(1+\kappa)\alpha(\omega)}<\frac{2}{1-\kappa}$ since $\gamma(\omega)=\frac{1-\alpha(\omega)}{1+\alpha(\omega)}<\kappa$. Fix $p=p(\omega)$ such that
 $$\max\left\{1,\frac{2}{(1+\kappa)\alpha(\omega)}\right\}<p<\frac{2}{1-\kappa}.$$ Then \eqref{kernel-2} together with H\"{o}lder's inequality implies
\begin{eqnarray}\label{L1a}
&&\int_{\mathbb{D}\setminus D(z,R)}|\langle b_z^\omega, b_{\zeta}^\omega\rangle_{A_\omega^2}|\frac{\|B_z^\omega\|_{A_\omega^2}^{\kappa}}{\|B_\zeta^\omega\|_{A_\omega^2}^{\kappa}} d\lambda_\omega(\zeta)\nonumber\\
&\asymp&\big(\widehat{\omega}(z)(1-|z|)\big)^{\frac{1-\kappa}{2}}\int_{\mathbb{D}\setminus D(z,R)}\big|(1-\bar{z}\zeta)B_z^\omega(\zeta)\big|
\widehat{\omega}(\zeta)^{\frac{1+\kappa}{2}} \frac{\big(1-|\zeta|^2\big)^{\frac{1+\kappa}{2}}}{|1-\bar{z}\zeta|} d\lambda_\omega(\zeta)\nonumber\\
&\leq &\big(\widehat{\omega}(z)(1-|z|)\big)^{\frac{1-\kappa}{2}}\left(\int_{\mathbb{D}}\big|\big(1-\bar{z}\zeta\big)B_z^\omega(\zeta)\big|^p
\widehat{\omega}(\zeta)^{\frac{p(1+\kappa)}{2}}  d\lambda_\omega(\zeta)\right)^{1/p} \nonumber\\
&&\qquad \qquad \qquad \qquad \qquad
 \cdot\left(\int_{\mathbb{D}\setminus \Delta(z,R')}\frac{(1-|\zeta|^2)^{\frac{p'(1+\kappa)}{2}}}{|1-\bar{z}\zeta|^{p'}} d\lambda_\omega(\zeta)\right)^{1/p'},
\end{eqnarray}
where $R=\log\frac{1+R'}{1-R'}$ and $ \Delta(z,R')=\{\zeta\in\mathbb{D}:|\varphi_z(\zeta)|<R'\}$.
Now define $$\nu(s)=\frac{\omega(s)}{\widehat{\omega}(s)}(1-s)^{\frac{p'(1+\kappa)}{2}-1},\quad 0\leq s<1.$$ It follows from \cite[Lemma 3]{PPR} that $\nu\in\mathcal{D}$ and $\widehat{\nu}(s)\asymp (1-s)^{\frac{p'(1+\kappa)}{2}-1}$.
By \cite[Lemma 3]{PR3}, there exists a constant $\varrho=\varrho(\nu)>0$ such that $\nu_{[-\rho]}\in \mathcal{D}$ and  $\widehat{\nu_{[-\rho]}}\asymp \widehat{\nu}_{[-\rho]}$ for every $0<\rho<\varrho$.
Choose $c=c(\omega,p)$ such that $0<c<\min\big\{\varrho,\frac{p'(1-\kappa)}{2}\big\}$. Notice that for $\zeta\in \mathbb{D}\setminus \Delta(z,R')$, we have $\frac{(1-|z|^2)(1-|\zeta|^2)}{|1-\bar{z}\zeta|^2}\leq 1-R'^2$. Then, by \cite[Proposition 5]{PRS}, we have
\begin{eqnarray}\label{L1b}
\int_{\mathbb{D}\setminus \Delta(z,R')}\frac{(1-|\zeta|^2)^{\frac{p'(1+\kappa)}{2}}}{|1-\bar{z}\zeta|^{p'}} d\lambda_\omega(\zeta)
&\lesssim & \frac{(1-R')^c}{(1-|z|)^c}\int_{\mathbb{D}} \frac{\nu_{[-c]}(\zeta)}{|1-\bar{z}\zeta|^{p'-2c}}dA(\zeta)\nonumber\\
&\asymp &\frac{(1-R')^c}{(1-|z|)^c}\int_{\mathbb{D}} \frac{1}{|1-\bar{z}\zeta|^{p'-2c}} \frac{\widehat{\nu_{[-c]}}(\zeta)}{1-|\zeta|}dA(\zeta)\nonumber\\
&\asymp &\frac{(1-R')^c}{(1-|z|)^c}\int_{\mathbb{D}} \frac{(1-|\zeta|)^{\frac{p'(1+\kappa)}{2}-2-c}}{|1-\bar{z}\zeta|^{p'-2c}} dA(\zeta)\nonumber\\
&\asymp& \frac{(1-R')^c}{(1-|z|)^{\frac{p'(1-\kappa)}{2}}},\quad z\in \mathbb{D}.
\end{eqnarray}
Moreover, let $\tau(s)=\widehat{\omega}(s)^{\frac{p(1+\kappa)}{2}}(1-s)^{-2}$ for all $0 \leq s<1.$ \eqref{Dcheck} and \eqref{Dhat1} show that $\tau\in \mathcal{D}$ and $\widehat{\tau}(s)\asymp\widehat{\omega}(s)^{\frac{p(1+\kappa)}{2}}(1-|s|)^{-1}$. Then \eqref{kernel-3}, \cite[Lemma 3]{PPR} and \eqref{Dcheck} imply
\begin{eqnarray}\label{L1c}
\int_{\mathbb{D}}|\big(1-\bar{z}\zeta\big)B_z^\omega(\zeta)|^p
\widehat{\omega}(\zeta)^{\frac{p(1+\kappa)}{2}}  d\lambda_\omega(\zeta)
&\lesssim& \int_0^{|z|}\left(\int_{r}^{1} \frac{\omega(s)\widehat{\tau}(s)}{\widehat{\omega}(s)}ds\right)\frac{1}{\widehat{\omega}(r)^p} dr+1\nonumber\\
&\lesssim& \int_0^{|z|} \frac{dr}{\widehat{\omega}(r)^{\frac{p(1-\kappa)}{2}}(1-r)}  +1\nonumber\\
&\lesssim &\bigg(\frac{(1-|z|)^{\alpha(\omega)}}{\widehat{\omega}(z)}\bigg)^{\frac{p(1-\kappa)}{2}}\int_0^{|z|} \frac{dr}{(1-r)^{1+\frac{p(1-\kappa)\alpha(\omega)}{2}}}  +1\nonumber\\
&\lesssim &\frac{1}{\widehat{\omega}(z)^{\frac{p(1-\kappa)}{2}}}, \quad{z\in\mathbb{D}}.
\end{eqnarray}
Notice that $R'\rightarrow 1$ when $R\rightarrow \infty$.
Combining \eqref{L1a}, \eqref{L1b} with \eqref{L1c} and then letting $R\rightarrow \infty$, we complete the proof.
\end{proof}
Now we show that every Toeplitz operator with bounded symbol is weakly localized.
\begin{proposition}\label{TinWL}
Let $\omega\in\mathcal{D}$ and $1<p<\infty$. For $f\in L^\infty$, each Toeplitz operator $T_f^\omega\in \mathrm{WL}_\omega^p$.
\end{proposition}
\begin{proof}
Since $(T_f^\omega)^*=T_{\bar{f}}^\omega$, the corresponding conditions for $(T_f^\omega)^*$ can be established in the same way as those for $T_f^\omega$.
We now prove
\begin{equation}\label{TinWL1}
\sup_{z\in\mathbb{D}} \int_{\mathbb{D}\setminus D(z,r)}\left|\langle T_f^\omega b_z^\omega, b_{\zeta}^\omega\rangle_{A_\omega^2}\right|
\frac{\|B_z^\omega\|_{A_\omega^2}^{\kappa}}{\|B_\zeta^\omega\|_{A_\omega^2}^{\kappa}}d\lambda_\omega(\zeta)\rightarrow 0, \quad r\rightarrow\infty.
\end{equation}
Indeed, by Fubini's theorem, we have
$$
\langle T_f^\omega b_z^\omega, b_{\zeta}^\omega\rangle_{A_\omega^2}
= \int_{\mathbb{D}} \langle b_u^\omega,b_\zeta^\omega\rangle_{A_\omega^2} \langle b_z^\omega, b_u^\omega\rangle _{A_\omega^2} f(u)\|B_u^\omega\|_{A_\omega^2}^2\omega(u)dA(u),\quad z,\zeta\in\mathbb{D},\quad f\in L^\infty.
$$
It follows from \eqref{kernel-2} and Fubini's theorem that
$$
\begin{aligned}
&\int_{\mathbb{D}\setminus D(z,r)}\left|\langle T_f^\omega b_z^\omega, b_{\zeta}^\omega\rangle_{A_\omega^2}\right| \frac{\|B_z^\omega\|_{A_\omega^2}^{\kappa}}{\|B_\zeta^\omega\|_{A_\omega^2}^{\kappa}}d\lambda_\omega(\zeta)\\
\lesssim&\|f\|_{\infty} \int_{\mathbb{D}}\left(\int_{\mathbb{D}\setminus D(z,r)}   \left|\langle b_u^\omega,b_\zeta^\omega\rangle_{A_\omega^2}\right|\frac{\|B_z^\omega\|_{A_\omega^2}^{\kappa}}{\|B_\zeta^\omega\|_{A_\omega^2}^{\kappa}} d\lambda_\omega(\zeta)\right)
\left|\langle b_z^\omega, b_u^\omega\rangle _{A_\omega^2}\right|\frac{\omega(u)}{\widehat{\omega}(u)(1-|u|)}dA(u)\\
=&\|f\|_{\infty}\big(I_1(z)+I_2(z)\big),
\end{aligned}
$$
where
$$
I_1(z)=\int_{D(z,\frac{r}{2})}\left(\int_{\mathbb{D}\setminus D(z,r)}   \left|\langle b_u^\omega,b_\zeta^\omega\rangle_{A_\omega^2}\right|\frac{\|B_z^\omega\|_{A_\omega^2}^{\kappa}}{\|B_\zeta^\omega\|_{A_\omega^2}^{\kappa}}  d\lambda_\omega(\zeta)\right)
\left|\langle b_z^\omega, b_u^\omega\rangle _{A_\omega^2}\right|d\lambda_\omega(u),
$$ and
$$
I_2(z)=\int_{\mathbb{D}\setminus D(z,\frac{r}{2})}\left(\int_{\mathbb{D}\setminus D(z,r)}   |\langle b_u^\omega,b_\zeta^\omega\rangle_{A_\omega^2}|\frac{\|B_z^\omega\|_{A_\omega^2}^{\kappa}}{\|B_\zeta^\omega\|_{A_\omega^2}^{\kappa}}  d\lambda_\omega(\zeta)\right)
\left|\langle b_z^\omega, b_u^\omega\rangle _{A_\omega^2}\right|d\lambda_\omega(u).
$$
Lemma \ref{kb} shows that
$$
\begin{aligned}
I_1(z)
\leq& \int_{\mathbb{D}}\left(\int_{\mathbb{D}\setminus D(u, \frac{r}{2})}   |\langle b_u^\omega,b_\zeta^\omega\rangle_{A_\omega^2}|\frac{\|B_u^\omega\|_{A_\omega^2}^{\kappa}}{\|B_\zeta^\omega\|_{A_\omega^2}^{\kappa}}d\lambda_\omega(\zeta)\right)
\left|\langle b_z^\omega,b_u^\omega\rangle_{A_\omega^2}\right|
\frac{\|B_z^\omega\|_{A_\omega^2}^{\kappa}}{\|B_u^\omega\|_{A_\omega^2}^{\kappa}}d\lambda_\omega(u)\\
\lesssim & \sup_{u\in\mathbb{D}} \int_{\mathbb{D}\setminus D(u, \frac{r}{2})}   \left|\langle b_u^\omega,b_\zeta^\omega\rangle_{A_\omega^2}\right|\frac{\|B_u^\omega\|_{A_\omega^2}^{\kappa}}{\|B_\zeta^\omega\|_{A_\omega^2}^{\kappa}}  d\lambda_\omega(\zeta),\quad {z\in\mathbb{D}}.
\end{aligned}
$$
Lemma \ref{kb} also yields
$$
\begin{aligned}
I_2(z)
\leq & \int_{\mathbb{D}\setminus D(z,\frac{r}{2})}\left(\int_{\mathbb{D}}   \left|\langle b_u^\omega,b_\zeta^\omega\rangle_{A_\omega^2}\right|\frac{\|B_u^\omega\|_{A_\omega^2}^{\kappa}}{\|B_\zeta^\omega\|_{A_\omega^2}^{\kappa}}  d\lambda_\omega(\zeta)\right)
\left|\langle  b_z^\omega, b_u^\omega\rangle_{A_\omega^2}\right|\frac{\|B_z^\omega\|_{A_\omega^2}^{\kappa}}{\|B_u^\omega\|_{A_\omega^2}^{\kappa}}
d\lambda_\omega(u)\\
\lesssim& \sup_{z\in\mathbb{D}}\int_{\mathbb{D}\setminus D(z,\frac{r}{2})} \left|\langle  b_z^\omega, b_u^\omega\rangle _{A_\omega^2}\right|\frac{\|B_z^\omega\|_{A_\omega^2}^{\kappa}}{\|B_u^\omega\|_{A_\omega^2}^{\kappa}}d\lambda_\omega(u),\quad{z\in\mathbb{D}}.
\end{aligned}
$$
The estimates of the above two integrals, combined with Lemma \ref{kc}, show that \eqref{TinWL1} is satisfied, which completes the proof.
\end{proof}
\begin{lemma}\label{Bound}
Let $\omega\in\mathcal{D}$ and $1<p<\infty$. If $T\in\mathrm{WL}_\omega^p$, then $T$ can be extended as a bounded operator on $A_\omega^p$.
\end{lemma}
\begin{proof}
Let $f$ be in the span of the reproducing kernels. By the reproducing formula, we have
$$
|Tf(z)|=|\langle Tf, B_{z}^{\omega}\rangle_{A_\omega^{2}}|
\leq\int_{\mathbb{D}}|f(\zeta)|\left|\langle T^{*}B_{z}^{\omega}, B_{\zeta}^{\omega}\rangle_{A_\omega^{2}}\right|\omega(\zeta)dA(\zeta),\quad z\in\mathbb{D}.$$ 
Let $\delta$ be the one in \eqref{delta}
and $h(z)=\|B_z^\omega\|_{A_\omega^2}^{\frac{\delta}{pp'}},$ where ${z\in\mathbb{D}}.$ Then \eqref{kernel-2} together with \eqref{WL2} shows that
$$
\int_{\mathbb{D}}\left|\langle T^*B_z^\omega,B_\zeta^\omega\rangle_{A_\omega^2}\right| h(\zeta)^{p'}\omega(\zeta)dA(\zeta)
\asymp \|B_z^\omega\|_{A_\omega^2}^{\frac{\delta}{p}}\int_{\mathbb{D}}\left|\langle T^{*}b_{z}^{\omega}, b_{\zeta}^{\omega}\rangle_{A_\omega^{2}}\right|
\frac{{\|B_{z}^{\omega}\|}^{1-\frac{\delta}{p}}_{A_\omega^2}}{{\|B_{\zeta}^{\omega}\|}^{1-\frac{\delta}{p}}_{A_\omega^2}}
d\lambda_\omega(\zeta)
\lesssim h(z)^{p'}.
$$
Similarly, \eqref{WL1} yields
$$
\int_{\mathbb{D}}\left|\langle T^*B_z^\omega,B_\zeta^\omega\rangle_{A_\omega^2}\right| h(z)^{p}\omega(z)dA(z) \lesssim h(\zeta)^p,\quad{\zeta\in\mathbb{D}}.
$$
Then Schur's test combined with the density argument shows that $T$ is bounded on $A_\omega^p$.
\end{proof}
\begin{proposition}\label{algebra}
Let $\omega\in\mathcal{D}$ and $1<p<\infty$. Then $\mathrm{WL}_\omega^p$ is an algebra. Furthermore, $\mathrm{WL}_\omega^2$ is a $*$-algebra.
\end{proposition}
\begin{proof}
It follows immediately from the definition of weakly localized operators that $T\in\mathrm{WL_\omega^2}$ implies $T^* \in\mathrm{WL}_\omega^2$.
Next, we show that if $T,S \in \mathrm{WL}_\omega^p$, then $TS$ is also in $\mathrm{WL}_\omega^p$. It suffices to prove
\begin{equation}\label{alg1}
\lim_{r\rightarrow \infty} \sup_{z\in\mathbb{D}}\int_{\mathbb{D}\setminus D(z,r)}\left|{\langle TSb_{z}^{\omega},b_{\zeta}^{\omega}\rangle}_{A_\omega^{2}}\right|
\frac{\|B_{z}^{\omega}\|_{A_\omega^{2}}^{1-\frac{\delta}{p'}}}{\|B_{\zeta}^{\omega}\|_{A_\omega^2}^{1-\frac{\delta}{p'}}}
d\lambda_\omega(\zeta)=0,
\end{equation}
where $\delta$ is that of \eqref{delta}.
Indeed, \eqref{kernel-2} together with Fubini's theorem gives
$$
\begin{aligned}
&\int_{\mathbb{D}\setminus D(z,r)}\left|{\langle TSb_{z}^{\omega},b_{\zeta}^{\omega}\rangle}_{A_\omega^{2}}\right|
\frac{\|B_{z}^{\omega}\|_{A_\omega^{2}}^{1-\frac{\delta}{p'}}}{\|B_{\zeta}^{\omega}\|_{A_\omega^2}^{1-\frac{\delta}{p'}}}
d\lambda_\omega(\zeta)\\
\lesssim& \int_{\mathbb{D}\setminus D(z,r)}\left(\int_{\mathbb{D}}\left|\langle Sb_{z}^{\omega}, b_{u}^{\omega}\rangle_{A_\omega^{2}}\right|\left|{\langle T^{*}b_{\zeta}^{\omega}, b_{u}^{\omega}\rangle}_{A_\omega^{2}}\right|\frac{\omega(u)}{\widehat{\omega}(u)(1-|u|)}dA(u)\right)
\frac{\|B_{z}^{\omega}\|_{A_\omega^2}^{1-\frac{\delta}{p'}}}{\|B_{\zeta}^{\omega}\|_{A_\omega^2}^{1-\frac{\delta}{p'}}}
d\lambda_\omega(\zeta)\\
=&\int_{\mathbb{D}}\left(\int_{\mathbb{D}\setminus D(z,r)}\left|{\langle Tb_{u}^{\omega}, b_{\zeta}^{\omega}\rangle}_{A_\omega^{2}}\right|\frac{\|B_{z}^{\omega}\|_{A_\omega^2}^{1-\frac{\delta}{p'}}}
{\|B_{\zeta}^{\omega}\|_{A_\omega^2}^{1-\frac{\delta}{p'}}}d\lambda_\omega(\zeta)\right)
\left|\langle Sb_{z}^{\omega}, b_{u}^{\omega}\rangle_{A_\omega^{2}}\right|d\lambda_\omega(u),\quad z\in \mathbb{D}.
\end{aligned}
$$
For $z\in\mathbb{D}$, the above integral can be split into two parts
$$
J_1(z)=\int_{D(z,\frac{r}{2})}\left(\int_{\mathbb{D}\setminus D(z,r)}\left|{\langle Tb_{u}^{\omega}, b_{\zeta}^{\omega}\rangle}_{A_\omega^{2}}\right|\frac{\|B_{z}^{\omega}\|_{A_\omega^2}^{1-\frac{\delta}{p'}}}
{\|B_{\zeta}^{\omega}\|_{A_\omega^2}^{1-\frac{\delta}{p'}}}d\lambda_\omega(\zeta)\right)
\left|\langle Sb_{z}^{\omega}, b_{u}^{\omega}\rangle_{A_\omega^{2}}\right|d\lambda_\omega(u),
$$
and
$$
J_2(z)=\int_{\mathbb{D}\setminus D(z,\frac{r}{2})}\left(\int_{\mathbb{D}\setminus D(z,r)}\left|{\langle Tb_{u}^{\omega}, b_{\zeta}^{\omega}\rangle}_{A_\omega^{2}}\right|\frac{\|B_{z}^{\omega}\|_{A_\omega^2}^{1-\frac{\delta}{p'}}}{\|B_{\zeta}^{\omega}\|_{A_\omega^2}^{1-\frac{\delta}{p'}}}
d\lambda_\omega(\zeta)\right)
\left|\langle Sb_{z}^{\omega}, b_{u}^{\omega}\rangle_{A_\omega^{2}}\right|d\lambda_\omega(u).
$$
By an argument similar to that in the proof of Proposition \ref{TinWL}, \eqref{alg1} is satisfied, and hence we complete the proof.
\end{proof}
Let $\mathcal{T}_\omega^p$ denote the Toeplitz algebra on $A_\omega^{p}$ generated by Toeplitz operators induced by $L^{\infty}$ symbols. The following relationship between $\mathcal{T}_\omega^p$ and $\mathrm{WL}_\omega^p$ can be obtained by combining Propositions \ref{TinWL} and \ref{algebra}.
\begin{theorem}
Let $\omega\in \mathcal{D}$ and $1<p<\infty$. Then the Toeplitz algebra $\mathcal{T}_\omega^p$ is contained in the norm closure of $\mathrm{WL}_\omega^p$.
\end{theorem}
\subsection{Weakly localization and compactness}

In this subsection we first show that the weakly localized operators can be approximated by infinite sums of
well-localized pieces in $\mathcal{D}$-weighted Bergman spaces. Next we characterize the compactness
of the operators belonging to the norm closure of $\mathrm{WL}_\omega^p$ by certain vanishing property, which is further shown to be equivalent to the boundary vanishing of the Berezin transform. Finally, we present several characterizations of compact operators in the class $\mathrm{WL}_\omega^p$.
To this end, we need the following covering lemma from \cite{MW}.
\begin{lemma}\label{cover}
There exists a positive integer $N$ such that, for any $r>0$, there is a covering $\mathcal{F}_r=\left\{F_j\right\}$ of $\mathbb{D}$ by disjoint Borel sets satisfying the following two conditions:
\begin{itemize}
\item[(i)] Every point of $\mathbb{D}$ belongs to at most $N$ of the sets $G_j:=\left\{z \in \mathbb{D}: \beta(z, \zeta) \leq r, \forall\zeta \in F_j\right\}$;
\item[(ii)] $\operatorname{diam}_d F_j \leq 2 r$ for every $j$.
\end{itemize}
\end{lemma}

\begin{proposition}\label{lp}
Let $\omega\in\mathcal{D}$ and $1<p<\infty$. If $T$ is in the norm closure of $\mathrm{WL}_\omega^p$, then for every $\varepsilon>0$, there exists $r=r(T,\omega,p)>0$ sufficiently large such that for the covering $\mathcal{F}_r=\left\{F_j\right\}$ we have
$$\label{appro}
\|TP_\omega-\sum_{j}{M_{\chi_{F_{j}}}}TP_\omega{M_{\chi_{G_{j}}}}\|_{A_\omega^p\rightarrow L_\omega^p} <\varepsilon,
$$
where $M_u$ is the multiplication operator with the symbol $u$.
\end{proposition}
\begin{proof}
Assume $T\in \mathrm{WL}_\omega^p$. Then for every $\varepsilon>0$, there exists $r=r(T,\omega,p)>0$ large enough such that $$
\sup_{z\in\mathbb{D}}\int_{\mathbb{D}\setminus D(z,r)}\left|\langle Tb_{z}^{\omega}, b_{\zeta}^{\omega}\rangle_{A_\omega^{2}}\right|
\frac{{\|B_{z}^{\omega}\|}^{1-\frac{\delta}{p'}}_{A_\omega^2}}{{\|B_{\zeta}^{\omega}\|}^{1-\frac{\delta}{p'}}_{A_\omega^2}}
d\lambda_\omega(\zeta)<\varepsilon
$$
and
$$
\sup_{z\in\mathbb{D}}\int_{\mathbb{D}\setminus D(z,r)}\left|\langle T^*b_{z}^{\omega}, b_{\zeta}^{\omega}\rangle_{A_\omega^{2}}\right|
\frac{{\|B_{z}^{\omega}\|}^{1-\frac{\delta}{p}}_{A_\omega^2}}{{\|B_{\zeta}^{\omega}\|}^{1-\frac{\delta}{p}}_{A_\omega^2}}
d\lambda_\omega(\zeta)<\varepsilon,
$$
where $\delta$ is the one in \eqref{delta}.
By Lemma \ref{cover}, for any given $r$, there exists a disjoint covering $\mathcal{F}_r = \{F_j\}$ of $\mathbb{D}$ such that the associated collection $\{G_j\}$ has finite multiplicity.
Now let $$S_\omega f(z)= TP_\omega f(z)-\sum_{j}({M_{\chi_{F_{j}}}}TP_\omega{M_{\chi_{G_{j}}}}f)(z),\quad f\in A_{\omega}^{p}, \quad z\in\mathbb{D}.$$
For $f\in A_{\omega}^{p}$, Fubini's theorem implies
\begin{eqnarray}\label{FR}
({M_{\chi_{F_{j}}}}TP_\omega{M_{\chi_{G_{j}}}}f)(z)
&=& {\chi_{F_{j}}}(z) \langle P_\omega({M_{\chi_{G_{j}}}}f), T^*B_z^\omega\rangle_{A_\omega^2}\nonumber\\
&=& {\chi_{F_{j}}}(z)\int_{G_{j}}f(\zeta)\overline{\langle T^{*}B_{z}^{\omega}, B_{\zeta}^{\omega}\rangle_{A_\omega^{2}}}\omega(\zeta)dA(\zeta),\quad z\in\mathbb{D}.
\end{eqnarray}
For any $z\in\mathbb{D}$, there exists $j_0$ such that $z\in F_{j_{0}}$. Moreover, Lemma \ref{cover}(i) shows that $D(z,r)\subset G_{j_0}$ whenever $z\in F_{j_0}$. Then it follows that
$$
\begin{aligned}
|S_\omega f(z)|
=&
\bigg|\int_{\mathbb{D}\setminus G_{j_{0}}}f(\zeta)\overline{\langle T^{*}B_{z}^{\omega}, B_{\zeta}^{\omega}\rangle_{A_\omega^{2}}}\omega(\zeta)dA(\zeta)\bigg| \\
\leq& \int_{\mathbb{D}}\chi_{\mathbb{D}\setminus D(z,r)}|f(\zeta)|\left|\langle T^{*}B_{z}^{\omega}, B_{\zeta}^{\omega}\rangle_{A_\omega^{2}}\right|\omega(\zeta)dA(\zeta).
\end{aligned}
$$
Let $h(z)={\|B_{z}^{\omega}\|}_{A_\omega^2}^{\frac{\delta}{pp'}}$ and $H(z,\zeta)=\chi_{\mathbb{D}\setminus D(z,r)}(\zeta)\left|\langle T^{*}B_{z}^{\omega}, B_{\zeta}^{\omega}\rangle_{A_\omega^{2}}\right|$, where $z,\zeta \in \mathbb{D}$. Then \eqref{kernel-2} yields
$$
\begin{aligned}
 \int_{\mathbb{D}}H(z,\zeta){h(\zeta)}^{p'}\omega(\zeta)dA(\zeta)
\asymp {\|B_{z}^{\omega}\|}^{\frac{\delta}{p}}_{A_\omega^2}\int_{\mathbb{D}\setminus D(z,r)}\left|\langle T^{*}b_{z}^{\omega}, b_{\zeta}^{\omega}\rangle_{A_\omega^{2}}\right|
\frac{{\|B_{z}^{\omega}\|}^{1-\frac{\delta}{p}}_{A_\omega^2}}{{\|B_{\zeta}^{\omega}\|}^{1-\frac{\delta}{p}}_{A_\omega^2}}
d\lambda_\omega(\zeta)
< \varepsilon {h(z)}^{p'}.
\end{aligned}
$$
Similarly, we have
$$
\int_{\mathbb{D}}H(z,\zeta){h(z)}^{p}\omega(z)dA(z)
< \varepsilon {h(\zeta)}^{p'},\quad{\zeta\in\mathbb{D}}.
$$
Then Schur's test and a standard approximation argument complete the proof.
\end{proof}
Let $\mathcal{K}\big(A_\omega^p\big)$ be the set of all compact operators on $A_\omega^p$, and denote
$$
\|T\|_{e,A_\omega^p}=\inf\bigg\{\|T-K\|_{A_\omega^p}:K\in\mathcal{K}\big(A_\omega^p\big)  \bigg\}.
$$
\begin{theorem}\label{seg}
Let $\omega\in\mathcal{D}$ and $1<p<\infty$. If $T$ is in the norm closure of $\mathrm{WL}_\omega^p$, then there exists $r=r(T,\omega,p)>0$ such that
\begin{equation}\label{Tcv}
\|T\|_{e,A_\omega^p} \lesssim \limsup_{|z|\rightarrow 1^-} \left(\sup_{\zeta\in D(z,r)} \left|\langle Tb_{p,z}^\omega,b_{p',\zeta}^\omega \rangle_{A_\omega^2}\right|\right).
\end{equation}
\end{theorem}

\begin{proof}
By Proposition \ref{lp}, there exists $r=r(T,\omega,p)> 0$ such that for the covering $\mathcal{F}_r = \{F_j\}$ associated with $r$, we have
$$
\|T P_\omega - \sum_j M_{\chi_{F_j}} T P_\omega M_{\chi_{G_j}}\|_{A_\omega^p \to L_\omega^p} < \frac{\|T\|_{e,A_\omega^p}}{2\|P_\omega\|_{L_\omega^p\rightarrow L_\omega^p}}.
$$
Since $T=P_\omega T P_\omega$ on $A_\omega^p$ and $\sum_{j < m} M_{\chi_{F_j}} T P_\omega M_{\chi_{G_j}}$ is compact for every $m \in \mathbb{N}$, using the fact that $P_\omega: L_\omega^p \rightarrow L_\omega^p$ is bounded,  we obtain
$$
\begin{aligned}
\|T\|_{e,A_\omega^p}
&\leq \left\|P_\omega T P_\omega -P_\omega \bigg(\sum_{j < m} M_{\chi_{F_j}} T P_\omega M_{\chi_{G_j}}\bigg)\right\|_{A_\omega^p \to A_\omega^p} \\
&\leq \|P_\omega\|_{L_\omega^p\rightarrow L_\omega^p} \left(\left\|T P_\omega - \sum_j M_{\chi_{F_j}} T P_\omega M_{\chi_{G_j}}\right\|_{A_\omega^p \to L_\omega^p} + \|T_m\|_{A_\omega^p \to L_\omega^p}\right) \\
&< \frac{1}{2} \|T\|_{e,A_\omega^p} + \|P_\omega\|_{L_\omega^p\rightarrow L_\omega^p}\|T_m\|_{A_\omega^p \to L_\omega^p},
\end{aligned}
$$
where
$$
T_m = \sum_{j \geq m} M_{\chi_{F_j}} T P_\omega M_{\chi_{G_j}}.
$$
Hence it suffices to show that there exists a constant $C(\omega,r,p)> 0$ such that
\begin{equation}\label{se}
\limsup_{m \to \infty} \left\|T_m\right\|_{A_\omega^p \to L_\omega^p} \leq C \limsup_{|z| \to 1^{-}} \left(\sup_{\zeta \in D(z, r)} \left| \left\langle T b_{p,z}^\omega, b_{p',\zeta}^\omega \right\rangle_{A_\omega^2} \right|\right) + \frac{\|T\|_{e,A_\omega^p}}{4\|P_\omega\|_{L_\omega^p\rightarrow L_\omega^p}} .
\end{equation}
To this end, let
$$
g_{j,\omega}(z)= \frac{ P_\omega M_{\chi_{G_j}} f (z)}{\| M_{\chi_{G_j}} f \|_{L_\omega^p} },\quad z\in\mathbb{D}.
$$
Then for $f\in A_\omega^p$, the finite overlap property of $G_j$ yields
$$
\begin{aligned}
\left\|T_m f\right\|_{L_\omega^p}^p &= \sum_{j \geq m}\| M_{\chi_{F_j}} T P_\omega M_{\chi_{G_j}} f \|_{L_\omega^p}^p \\
&= \sum_{j \geq m} \frac{ \| M_{\chi_{F_j}} T P_\omega M_{\chi_{G_j}} f \|_{L_\omega^p}^p }{ \| M_{\chi_{G_j}} f \|_{L_\omega^p}^p } \| M_{\chi_{G_j}} f \|_{L_\omega^p}^p \\
&\leq N \sup_{j \geq m} \| M_{\chi_{F_j}} T g_{j,\omega} \|_{L_\omega^p}^p\|f\|^p_{A_\omega^p}.
\end{aligned}
$$
It follows that
\begin{equation}\label{se3}
\limsup_{m \to \infty} \left\|T_m\right\|_{A_\omega^p \to L_\omega^p} \lesssim\limsup_{j \to \infty} \sup_{ \|f\|_{A_\omega^p} \leq 1}\| M_{\chi_{F_j}} T g_{j,\omega} \|_{L_\omega^p}.
\end{equation}
Now choose a sequence $\{f_j\} \subset A_\omega^p$ with $\|f_j\|_{A_\omega^p} \leq 1$ such that
\begin{eqnarray}\label{se1}
&&\limsup_{j \to \infty} \sup_{ \|f\|_{A_\omega^p} \leq 1 }  \| M_{\chi_{F_j}} T g_{j,\omega} \|_{L_\omega^p} \nonumber  \\
&\leq& \limsup_{j \to \infty} \bigg\| M_{\chi_{F_j}} T \frac{ P_\omega M_{\chi_{G_j}} f_j }{ \| M_{\chi_{G_j}} f_j \|_{L_\omega^p} } \bigg\|_{L_\omega^p}+\frac{\|T\|_{e,A_\omega^p}}{4\|P_\omega\|_{L_\omega^p\rightarrow L_\omega^p}}.
\end{eqnarray}
By an argument similar to that used in \eqref{FR}, we obtain
$$
\begin{aligned}
(TP_\omega M_{\chi_{G_j}} f_j)(z)
=&\int_{G_j} \overline{T^*B_z^\omega(\zeta)} f_j(\zeta)\omega(\zeta)dA(\zeta),\quad{z\in\mathbb{D}}.
\end{aligned}
$$
Then \eqref{kernel-2} shows that for every $j$,
\begin{eqnarray}\label{se2}
&&\bigg\| M_{\chi_{F_j}} T \frac{ P_\omega M_{\chi_{G_j}} f_j }{ \| M_{\chi_{G_j}} f_j \|_{L_\omega^p} } \bigg\|_{L_\omega^p}^p \nonumber\\
&\lesssim&\int_{F_j} \left(\frac{\int_{G_j} \left|\langle f_j, b_{p',\zeta}^\omega \rangle_{A_\omega^2}\right| |Tb_{p,\zeta}^\omega(z)|\|B_\zeta^\omega\|_{A_\omega^p} \|B_\zeta^\omega\|_{A_\omega^{p'}}\omega(\zeta) dA(\zeta)}{\left(\int_{G_j} |f_j(u)|^p\omega(u)dA(u)\right)^{1/p}} \right)^p \omega(z)dA(z)\nonumber \\
&\asymp& \int_{F_j} \left( \int_{G_j} |  H_{j,\omega}(\zeta) | \left| \langle T b_{p,\zeta}^\omega, b_{p',z}^\omega \rangle_{A_\omega^2} \right| d\lambda_\omega(\zeta) \right)^p d\lambda_\omega(z),
\end{eqnarray}
where
$$
H_{j,\omega}(\zeta)
=\frac{ \langle f_j, b_{p',\zeta}^\omega \rangle_{A_\omega^2} }{ \left( \int_{G_j} | \langle f_j, b_{p',u}^\omega \rangle_{A_\omega^2}|^p d\lambda_\omega(u) \right)^{1/p} },\quad{\zeta\in\mathbb{D}}.
$$
It follows from the proof of \cite[Proposition 2.11]{MW} that there exists $z_j\in\mathbb{D}$ such that $G_j\subset D(z_j,4r)$ for every $j$. Then \eqref{Dhat1} gives
$$
\lambda_\omega(G_j)
\leq \int_{D(z_j,4r)}\frac{\omega(\zeta)}{\widehat{\omega}(\zeta)(1-|\zeta|)}dA(\zeta)\asymp \frac{\omega\big(D(z_j,4r)\big)}{\widehat{\omega}(z_j)(1-|z_j|)}\lesssim 1, \quad j\in\mathbb{N}.
$$
Furthermore, it follows from Lemma \ref{cover} that $\beta(z,\zeta)\leq 3r$ for $z\in F_j$ and $\zeta \in G_j$. H\"{o}lder's inequality then yields
$$
\begin{aligned}
& \limsup_{j \to \infty} \int_{F_j} \left( \int_{G_j} |  H_{j,\omega}(\zeta) | \left| \langle T b_{p,\zeta}^\omega, b_{p',z}^\omega \rangle_{A_\omega^2}\right| d\lambda_\omega(\zeta) \right)^p d\lambda_\omega(z)\\
\leq&  \left(\limsup_{|z| \to 1^{-}}\left( \sup_{\zeta \in D(z, 3r)} \left|\langle T b_{p,z}^\omega, b_{p',\zeta}^\omega\rangle_{A_\omega^2}\right| ^p\right)\right) \cdot \left(\sup_j \left(\lambda_\omega(G_j) \right)^p \int_{G_j}  | H_{j,\omega}(\zeta)|^p d\lambda_\omega(\zeta)\right) \\
\lesssim&  \limsup_{|z| \to 1^{-}} \left(\sup_{\zeta \in D(z, 3r)} \left| \langle T b_{p,z}^\omega, b_{p',\zeta}^\omega\rangle_{A_\omega^2}\right|^p\right).
\end{aligned}
$$
This combined with \eqref{se3}, \eqref{se1} and \eqref{se2} implies \eqref{se}, which completes the proof.
\end{proof}
We then apply the affine transform to show that the vanishing behaviour of the quantity in the right-hand side of \eqref{Tcv} can be characterized by the boundary vanishing of the Berezin transform.
\begin{lemma}\label{equivalent}
 Let $\omega \in \widehat{\mathcal D}$, $1<p<\infty$ and $T$ be a bounded operator on ${A_\omega^p}$. Then
$
\lim\limits_{|z|\to 1^-}\widetilde{T}(z)=0
$ if and only if \begin{equation}\label{Berzin}
 \limsup_{|z|\rightarrow 1^-} \left( \sup_{\zeta\in D(z,r)} \left|\langle Tb_{p,z}^\omega,b_{p',\zeta}^\omega \rangle_{A_\omega^2}\right|\right)=0, \end{equation} for some (or equivalently for all) $r>0$.
\end{lemma}
\begin{proof}
Assume $\widetilde{T}(z)\to 0$ as $|z|\to 1^-$, but \eqref{Berzin} fails for some fixed $r>0$. Then there exist $\varepsilon>0$,
$z_j=\rho_j e^{i\theta_j}$ satisfying $\rho_j\to1^-$ as $j\to \infty$, and $\zeta_j\in D(z_j,r)$ such that
\begin{equation}\label{contray}
\left|
\left\langle
 T b_{p,z_j}^\omega, b^\omega_{p',\zeta_j}
\right\rangle_{A_\omega^2}
\right|
\ge\varepsilon.
\end{equation}
Fix a constant $s$ such that $r<s<\infty.$ Write \[
a_{j,s}
=
\frac{1-(\tanh s)^2}
{1-(\tanh s)^2\,\rho_j^2}\,z_j,
\qquad
R_{j,s}
=
\frac{(\tanh s)^2\,(1-\rho_j^2)}
{1-(\tanh s)^2\,\rho_j^2}.
\] Now let $\Psi_j(u)=a_{j,s}+e^{i\theta_j}R_{j,s}u$ for all $u\in\D.$
Then $\Psi_j(\D)=D(z_j,s)$ and $ \Psi_j(\rho_j\tanh s)=z_j.$ Moreover,
note that \[
\xi_j
=\Psi_j^{-1}(\zeta_j)=
e^{-i\theta_j}
\frac{\zeta_j-a_{j,s}}{R_{j,s}}, \quad \zeta_j\in D(z_j,r).
\]
It follows that
\begin{align*}
|\xi_j|
&\le
\frac{|\zeta_j-a_{j,r}|+|a_{j,r}-a_{j,s}|}
{R_{j,s}}<
\frac{R_{j,r}+|a_{j,r}-a_{j,s}|}
{R_{j,s}}
=\frac{\rho_j(\tanh s)^2+\tanh r}
{\tanh s\,(1+\rho_j\tanh r)}
< c,
\end{align*}
where $c=\frac{(\tanh s)^2+\tanh r}
{\tanh s\,(1+\tanh r)}$.
Set $\iota=\iota(r,s)=\max\left\{\tanh s,c\right\}$ and choose $\delta \in \left(0, \frac{1-\iota}{\iota}\right)$. Let
$$
\Omega_\delta
=
\left\{u\in\mathbb C: \inf_{t\in[-1,1]} |u-t|<\delta\right\},
$$
and define
$$\label{convex}
\psi_j(u)
=
\frac{\rho_j\tanh s+\xi_j}{2}
+
\frac{\xi_j-\rho_j\tanh s}{2} u, \quad u\in \Omega_\delta.
$$
Observe that $|\psi_j(t)|\le \iota$ whenever $-1\le t\le1$, and furthermore, $|\psi_j'|\le \iota.$
For $u\in\Omega_\delta$, there exists $t_0\in[-1,1]$ such that $|u-t_0|<\delta$. Then we have
\[
|\psi_j(u)|
\le
|\psi_j(t_0)|+|\psi_j'|\,|u-t_0|
<\iota(1+\delta)<1,
\]
which implies $\psi_j(\Omega_\delta)\subset\D.$
Now we can define a family $\{F_j\}\subset H(\Omega_\delta)$ by
$$\label{help}
F_j(u)
=\frac{1}{\|B_{z_j}^\omega\|_{A_\omega^p}
\|B_{z_j}^\omega\|_{A_\omega^{p'}}}\left\langle
T B_{\Psi_j\big(\psi_j(-\overline{u})\big)}^\omega,
B_{\Psi_j\big(\psi_j(u)\big)}^\omega
\right\rangle_{A_\omega^2},
\quad u\in\Omega_\delta.
$$
Since $\Psi_j(\D)=D(z_j,s)$, then H\"older's inequality together with \eqref{kernel-2} and \eqref{Dhat1} implies that
$$
|F_j(u)|
\le
\|T\|_{A_\omega^p\rightarrow A_\omega^p}
\frac{1}{\|B_{z_j}^\omega\|_{A_\omega^p}
\|B_{z_j}^\omega\|_{A_\omega^{p'}}}
\bigg\|B_{\Psi_j\big(\psi_j(-\overline{u})\big)}^\omega\bigg\|_{A_\omega^p}
\left\|B_{\Psi_j\big(\psi_j(u)\big)}^\omega\right\|_{A_\omega^{p'}}
\lesssim\|T\|_{A_\omega^p\rightarrow A_\omega^p},
\quad u\in\Omega_\delta.
$$
Thus $\{F_j\}$ is a uniformly bounded family in $H(\Omega_\delta)$. By Montel's theorem, there exist a subsequence, still denoted by $\{F_j\}$, and a function $F\in H(\Omega_\delta)$ such that
$ F_j\rightarrow F$ uniformly on compact subsets of $\Omega_\delta.$

Moreover, write
$u_\ell=\frac{i\delta}{\ell+1}$ for $\ell\geq 1.$
Note that $u_\ell\in\Omega_\delta$. Since $\psi_j(u_\ell)\in\D$, we have $\Psi_j(\psi_j(u_\ell))\in D(z_j,s)$ and $
|\Psi_j(\psi_j(u_\ell))|\rightarrow1^-$ as $j\to\infty$.
Then for each fixed $\ell$, using the assumption that $\widetilde{T}$ vanishes on the boundary, along with \eqref{kernel-2} and \eqref{Dhat1}, we obtain
$$
F_j(u_\ell)
=\frac{1}{\|B_{z_j}^\omega\|_{A_\omega^p}
\|B_{z_j}^\omega\|_{A_\omega^{p'}}}
\widetilde T\bigg(\Psi_j\big(\psi_j(u_\ell)\big)\bigg)\left\|B_{\Psi_j\big(\psi_j(u_\ell)\big)}^\omega\right\|_{A_\omega^2}^2
\asymp \widetilde T\bigg(\Psi_j\big(\psi_j(u_\ell)\big)\bigg)\to 0, \quad j\to \infty.
$$
It follows that $F(u_\ell)=0,$ for each $\ell\ge1.$ This together with the fact that $u_\ell\to 0$ as $\ell\to \infty$ and the continuity gives $F(0)=0$.
Therefore, we conclude that $F\equiv0$ on $\Omega_\delta.$

However, combining \eqref{kernel-2}, \eqref{Dhat1} and \eqref{contray},  for $\zeta_j\in D(z_j,r)$ we have
$$
|F_j(1)|
=
\frac{1}{\|B_{z_j}^\omega\|_{A_\omega^p}
\|B_{z_j}^\omega\|_{A_\omega^{p'}}}\left|\left\langle
T B_{\Psi_j(\rho_j \tanh s)}^\omega,
B_{\Psi_j(\xi_j)}^\omega
\right\rangle_{A_\omega^2}\right|
=\left|
\left\langle
T b_{p,z_j}^\omega, b^\omega_{p',\zeta_j}
\right\rangle_{A_\omega^2}
\right|
\frac{
\|B_{\zeta_j}^\omega\|_{A_\omega^{p'}}
}{
\|B_{z_j}^\omega\|_{A_\omega^{p'}}}\gtrsim\varepsilon,
$$
which contradicts $F(1)=0$. This proves \eqref{Berzin}. Then we complete the proof since the converse direction follows immediately.
\end{proof}
\color{black}
\vspace{8pt}
\noindent \textbf{Proof of Theorem \ref{main}.} The sufficiency follows directly from Theorem \ref{seg} and Lemma \ref{equivalent}. For the necessity,
by \cite[Lemma 10]{PRS}, $b_{p,z}^\omega$ converges weakly to $0$ in~$A_\omega^p$ as $|z|\rightarrow 1^-$. Therefore \eqref{kernel-2} yields $\widetilde{T}(z)\lesssim \|T b_{p,z}^\omega\|_{A_\omega^p}\rightarrow 0,$ as $ |z|\rightarrow 1^-,$
which completes the proof.
\qed

\vspace{8pt}

Recall that for $0<p,q<\infty$ and a positive Borel measure $\mu$ on $\mathbb{D}$, $\mu$ is called a $q$-Carleson measure for $A_\omega^p$ if the identity operator $I:A_\omega^p\rightarrow L_\mu^q$ is bounded.
The following result shows that every compact operator acting on the $\mathcal{D}$-weighted Bergman space belongs to the norm closure of $\mathrm{WL}_\omega^p$.
\begin{proposition}\label{comWL}
Let $\omega \in {\mathcal{D}}$ and $1<p<\infty$. If a linear bounded operator $T$ is compact on~$A_\omega^p$, then $T$ belongs to the norm closure of $\mathrm{WL}_\omega^p$.
\end{proposition}
\begin{proof}
For $f\in A_\omega^p$ and $g\in A_\omega^{p'}$, let $(f\otimes g)h=\langle h,g \rangle_{A_\omega^2} f$ for $h\in A_\omega^p$. We first assume $f$ and $g$ are polynomials.
Note that
\begin{equation}\label{fr2}
(1\otimes 1)h(z)
=\langle h,1 \rangle_{A_\omega^2}
=2\omega_1 h(0),\quad h\in A_\omega^p.
\end{equation}
Let $\delta_0$ be the Dirac measure with mass concentrated at $0$.
Then the Toeplitz operator induced by $\delta_0$ satisfies
$$
T^\omega_{\delta_0}h(z)=\int_{\mathbb{D}} h(\zeta)\overline{B_z^\omega(\zeta)} d\delta_0(\zeta)= \frac{1}{2\omega_1} h(0).
$$
This together with \eqref{fr2} gives $(1\otimes 1)=(2\omega_1)^2T^\omega_{\delta_0}.$ Moreover,
since $\delta_0$ is a Carleson measure for $A_\omega^p$, by an argument similar to the proof of Proposition \ref{TinWL}, it follows that
$T^\omega_{\delta_0}\in \mathrm{WL}_\omega^p$, and hence $1\otimes 1\in \mathrm{WL}_\omega^p$.
Fubini's theorem yields
$$\label{fr1}
\begin{aligned}
T^\omega_f(1\otimes 1)T^\omega_g (h)(z)
=(f\otimes g)h (z),\quad z\in\mathbb{D} ,\quad h\in A_\omega^p.
\end{aligned}
$$
This combined with $1\otimes 1\in \mathrm{WL}_\omega^p$ and Proposition \ref{algebra} implies $f \otimes g\in \mathrm{WL}_\omega^p$.
By \cite[p. 69-70]{Wo} and a standard density argument, the proof is completed.
\end{proof}
Combining Theorem \ref{main} and Proposition \ref{comWL}, we obtain the following result, extending \cite[Theorem 9.5]{Su} to the $\mathcal{D}$-weight setting.
\begin{corollary}
Let $\omega\in\mathcal{D}$, $1<p<\infty$ and $T$ be a bounded linear operator on $A_\omega^p$. Then $T$ is compact on $A_\omega^p$ if and only if $T$ belongs to the norm closure of $\mathrm{WL}_\omega^p$ and $
\lim\limits_{|z|\to 1^-}\widetilde{T}(z)=0.
$
\end{corollary}
In the following, we provide several further characterizations about the compactness for the operators in $\mathrm{WL_\omega^p}$, which generalizes \cite[Proposition 3.2 and Theorem 3.3]{MSWW}.
\begin{proposition}\label{WLequ}
Let $\omega\in\mathcal{D}$ and $1<p<\infty$. If $T\in \mathrm{WL}_\omega^p$, then
the following statements are equivalent:
\begin{itemize}
\item[(i)] $T$ is compact on $A_\omega^p$;
\item [(ii)]$\lim\limits_{|z|\to 1^-}\widetilde{T}(z)=0$;
\item[(iii)]$\lim\limits_{|z|\rightarrow 1^-}\sup\limits_{\zeta\in D(z,r)}|\langle T b_{p,z}^\omega,b_{p',\zeta}^\omega \rangle_{A_\omega^2}|= 0$ for sufficiently large $r=r(T,\omega,p)>0$;
\item[(iv)]For $0<\delta<\min\{p(1-\gamma(\omega)), p'(1-\gamma(\omega))\}$,
$$\lim_{R\rightarrow\infty}\sup_{z\in\mathbb{D}}\int_{\mathbb{D}\setminus D(0,R)}\left|\langle Tb_z^\omega,  b_{\zeta}^\omega\rangle_{A_\omega^2}\right|\frac{\|B_z^\omega\|_{A_\omega^2}^{1-\frac{\delta}{p'}}}{\|B_\zeta^\omega\|_{A_\omega^2}^{1-\frac{\delta}{p'}}} d\lambda_\omega(\zeta)=0. $$
\end{itemize}
\end{proposition}
\begin{proof}
(iv) $\Rightarrow$ (i). For each $f\in A_\omega^p$ and $\|f\|_{A_\omega^p}\leq 1$,
$$
\langle Tf, b_{p',\zeta}^\omega \rangle_{A_\omega^2}
=\int_{\mathbb{D}}\langle  T b_{p,z}^\omega,b_{p',\zeta}^\omega \rangle_{A_\omega^2}\langle f, b_{p',z}^\omega\rangle_{A_\omega^2} \|B_z^\omega\|_{A_\omega^p}\|B_z^\omega\|_{A_\omega^{p'}}\omega(z)dA(z),\quad{\zeta\in\mathbb{D}}.
$$
Then by H\"{o}lder's inequality, \eqref{kernel-2} and \eqref{WL2}, we have
\begin{eqnarray}\label{equ1}
|\langle Tf, b_{p',\zeta}^\omega \rangle_{A_\omega^2}|^p
&\lesssim& \bigg( \int_{\mathbb{D}}|\langle  T b_{p,z}^\omega,b_{p',\zeta}^\omega \rangle_{A_\omega^2}||\langle f, b_{p',z}^\omega\rangle_{A_\omega^2}|d\lambda_\omega(z) \bigg)^p \nonumber\\
&\leq& \int_{\mathbb{D}}\left|\langle  T b_{p,z}^\omega,b_{p',\zeta}^\omega \rangle_{A_\omega^2}\right| \left|\langle f, b_{p',z}^\omega\rangle_{A_\omega^2}\right|^p \left(\frac{\|B_z^\omega\|_{A_\omega^2}^{\frac{2}{p}-\frac{\delta}{p}}}{\|B_\zeta^\omega\|_{A_\omega^2}^{\frac{2}{p}-\frac{\delta}{p}}}\right)^{\frac{p}{p'}}
d\lambda_\omega(z)\nonumber\\
&&
\cdot\bigg( \int_{\mathbb{D}}\left|\langle  T b_{p,z}^\omega,b_{p',\zeta}^\omega \rangle_{A_\omega^2}\right| \frac{\|B_\zeta^\omega\|_{A_\omega^2}^{\frac{2}{p}-\frac{\delta}{p}}}{\|B_z^\omega\|_{A_\omega^2}^{\frac{2}{p}-\frac{\delta}{p}}}
d\lambda_\omega(z)\bigg)^{\frac{p}{p'}}\nonumber\\
&\lesssim& \int_{\mathbb{D}}\left|\langle  T b_{p,z}^\omega,b_{p',\zeta}^\omega \rangle_{A_\omega^2}\right| \left|\langle f, b_{p',z}^\omega\rangle_{A_\omega^2}\right|^p \frac{\|B_z^\omega\|_{A_\omega^2}^{\frac{2}{p'}-\frac{\delta}{p'}}}{\|B_\zeta^\omega\|_{A_\omega^2}^{\frac{2}{p'}-\frac{\delta}{p'}}}
d\lambda_\omega(z).
\end{eqnarray}
For any given $\varepsilon>0$, there exists $R_1>0$ such that for $R>R_1$,
$$
\int_{\mathbb{D}\setminus D(0,R)}\left|\langle Tb_z^\omega,  b_{\zeta}^\omega\rangle_{A_\omega^2}\right|\frac{\|B_z^\omega\|_{A_\omega^2}^{1-\frac{\delta}{p'}}}{\|B_\zeta^\omega\|_{A_\omega^2}^{1-\frac{\delta}{p'}}} d\lambda_\omega(\zeta)<\varepsilon, \quad z\in\mathbb{D}.
$$
Combining \eqref{kernel-2}, \eqref{equ1} and Fubini's theorem, we obtain
$$
\begin{aligned}
\int_{\mathbb{D}\setminus D(0,R)}|Tf|^p\omega dA
&\asymp\int_{\mathbb{D}\setminus D(0,R)}\left|\langle Tf, b_{p',\zeta}^\omega \rangle_{A_\omega^2}\right|^p d\lambda_\omega(\zeta)\\
&\lesssim \int_{\mathbb{D}}\bigg(\int_{\mathbb{D}\setminus D(0,R)}\left|\langle  T b_{z}^\omega,b_{\zeta}^\omega \rangle_{A_\omega^2}\right|\frac{\|B_z^\omega\|_{A_\omega^2}^{1-\frac{\delta}{p'}}}{\|B_\zeta^\omega\|_{A_\omega^2}^{1-\frac{\delta}{p'}}}
d\lambda_\omega(\zeta)\bigg) \left|\langle f, b_{p',z}^\omega\rangle_{A_\omega^2}\right|^p d\lambda_\omega(z)\\
&< \varepsilon \int_{\mathbb{D}}\left|\langle f, b_{p',z}^\omega\rangle_{A_\omega^2}\right|^p d\lambda_\omega(z)\asymp\varepsilon \|f\|_{A_\omega^p}^p
\leq \varepsilon.
\end{aligned}
$$
Then we apply \cite[Theorem 1.8]{MSWW} to deduce that $\{Tf: f\in A_\omega^p, \|f\|_{A_\omega^p}\leq1\}$ is precompact, and thus $T$ is compact on $A_\omega^p$.

(iii) $\Rightarrow$ (iv). Given $\varepsilon>0$. By \eqref{WL3}, there exists sufficiently large $R_0=R_0(T,\omega,p)>0$ such that
\begin{equation}\label{equ2}
\int_{\mathbb{D}\setminus D(z,R_0)}\left|\langle Tb_z^\omega,  b_{\zeta}^\omega\rangle_{A_\omega^2}\right|\frac{\|B_z^\omega\|_{A_\omega^2}^{1-\frac{\delta}{p'}}}{\|B_\zeta^\omega\|_{A_\omega^2}^{1-\frac{\delta}{p'}}} d\lambda_\omega(\zeta)<\varepsilon,\quad z\in\mathbb{D}.
\end{equation}
Note that there exists $N_0>0$ such that
\begin{equation}\label{equ3}
\left|\langle Tb_{p,z}^\omega,b_{p',\zeta}^\omega\rangle_{A_\omega^2}\right|<\varepsilon
\end{equation}
for all $z\in \mathbb{D}\setminus D(0,N_0)$ and $\zeta\in D(z,R_0)$. Now let $R=R_0+N_0$ and write
$$
\begin{aligned}
\int_{\mathbb{D}\setminus D(0,R)}|\langle Tb_z^\omega,  b_{\zeta}^\omega\rangle_{A_\omega^2}|\frac{\|B_z^\omega\|_{A_\omega^2}^{1-\frac{\delta}{p'}}}{\|B_\zeta^\omega\|_{A_\omega^2}^{1-\frac{\delta}{p'}}} d\lambda_\omega(\zeta)
=&\int_{\mathbb{D}\setminus\big(D(0,R)\cup D(z,R_0)\big)}\left|\langle Tb_z^\omega,  b_{\zeta}^\omega\rangle_{A_\omega^2}\right|\frac{\|B_z^\omega\|_{A_\omega^2}^{1-\frac{\delta}{p'}}}{\|B_\zeta^\omega\|_{A_\omega^2}^{1-\frac{\delta}{p'}}} d\lambda_\omega(\zeta)\\
&+\int_{\big(\mathbb{D}\setminus D(0,R)\big)\cap D(z,R_0)}\left|\langle Tb_z^\omega,  b_{\zeta}^\omega\rangle_{A_\omega^2}\right|\frac{\|B_z^\omega\|_{A_\omega^2}^{1-\frac{\delta}{p'}}}{\|B_\zeta^\omega\|_{A_\omega^2}^{1-\frac{\delta}{p'}}}d\lambda_\omega(\zeta).
\end{aligned}
$$
The first term above is controlled by the left side of \eqref{equ2}. Since $\beta(z,0)>N_0$ for $\zeta\in \big(\mathbb{D}\setminus D(0,R)\big)\cap D(z,R_0)$, by combining \eqref{Dhat1}, \eqref{kernel-2} and \eqref{equ3}, we have
$$
\int_{\big(\mathbb{D}\setminus D(0,R)\big)\cap D(z,R_0)}\left|\langle Tb_z^\omega,  b_{\zeta}^\omega\rangle_{A_\omega^2}\right|\frac{\|B_z^\omega\|_{A_\omega^2}^{1-\frac{\delta}{p'}}}{\|B_\zeta^\omega\|_{A_\omega^2}^{1-\frac{\delta}{p'}}} d\lambda_\omega(\zeta)
\lesssim \varepsilon \lambda_\omega\big(D(z,R_0)\big)\lesssim \varepsilon,
$$
and hence (iv) is satisfied.

 Since $b_{p,z}^\omega$ converges weakly to $0$ in $A_\omega^p$ as $|z|\rightarrow 1^-$, we obtain (i) $\Rightarrow$ (ii).
The assertion (ii) $\Rightarrow$ (iii) $\Rightarrow$ (i) follows from Lemma \ref{equivalent} and Theorem \ref{seg}. This completes the proof.
\end{proof}

\subsection{Toeplitz operator with $\mathrm{BMO}_\omega^p$ symbol }
In 2003, Zorboska \cite{Zorboska} proved that the compactness of Toeplitz operators with $\mathrm{BMO}$ symbols in terms of the boundary behaviour of the Berezin transform. Sadeghi and Zorboska \cite{SZ} later showed that a Toeplitz operator with complex Borel measure symbol, whose total variation is a Carleson measure, is weakly localized on the Bergman space.

Based on the work \cite{SZ,Zorboska,Zorboska2}, in this subsection we characterize the compact Toeplitz operators with symbols in $\mathrm{BMO}_\omega^p$.
For $\omega\in \mathcal{D}$, it follows from \eqref{Dcheck} and \eqref{Dhat1} that there exists $r_0=r_0(\omega)>0$ such that
$$
\omega\big(D(z,r)\big)\asymp \widehat{\omega}(z)(1-|z|),\quad{z\in\mathbb{D}},
$$
for each fixed $r\geq r_0$.
For $1 \leq p<\infty$ and each fixed $r\geq r_0$, write
$$
\mathrm{MO}_{\omega, r}^p(f)(z)=\left(\frac{1}{\omega\big(D(z, r)\big)} \int_{D(z, r)}\left|f(\zeta)-\widehat{f}_{r, \omega}(z)\right|^p \omega(\zeta) \mathrm{d}
A(\zeta)\right)^{1/p},
$$
where
$$
\widehat{f}_{r, \omega}(z)=\frac{\int_{D(z, r)} f(\zeta) \omega(\zeta) \mathrm{d} A(\zeta)}{\omega\big(D(z, r)\big)}, \quad z \in \mathbb{D}.
$$
The space $\mathrm{BMO}_{\omega,  r}^p$ consists of $f \in L_\omega^p$ such that
$$
\|f\|_{\mathrm{BMO}_{\omega, r}^p}=\sup _{z \in \mathbb{D}} \Big(\mathrm{MO}_{\omega, r}^p(f)(z)\Big)<\infty.
$$
The reproducing kernel of the standard weighted Bergman space $A^2_{\alpha}$ is $K_z^\alpha(\zeta)={{(1-\bar{z}\zeta)}^{-(\alpha+2)}}$, and we denote
$
k_{z}^{\alpha,\omega}(\zeta)=K_z^\alpha(\zeta) / \|K_z^\alpha\|_{A_\omega^2}.
$
For $g\in L_\omega^1$, the Berezin-type transform is defined as
$$
B_\omega^\alpha(g)(z)=\langle gk_{z}^{\alpha,\omega}, k_{z}^{\alpha,\omega}\rangle_{L^2_\omega}, \quad{z\in\mathbb{D}}.
$$

Now we are in a position to give the following characterization concerning the compactness of the Toeplitz operator with a $\mathrm{BMO}_\omega^p$ symbol.
\begin{proposition}\label{TBMO}
Let $\omega\in \mathcal{D}$ and $1<p<\infty$. For $f\in \mathrm{BMO}_\omega^p$, if there exists $\alpha_0=\alpha_0(\omega)>0$ such that $B_\omega^\alpha(f)$ is bounded in $\mathbb{D}$ for any $\alpha>\alpha_0$, then $T_f^\omega$ is bounded on $A_\omega^p$. And furthermore, the following statements hold,
\begin{itemize}
\item[(i)]$T_f^\omega$ is weakly localized on $A_\omega^p$;
\item[(ii)]$T_f^\omega$ is compact on $A_\omega^p$ if and only if $
\lim\limits_{|z|\to 1^-}\widetilde{T}(z)=0 .
$
\end{itemize}
\end{proposition}
\begin{proof}
For $f\in \mathrm{BMO}_\omega^p$, since $B_\omega^\alpha(f)$ is bounded in $\mathbb{D}$ for any $\alpha>\alpha_0$, it follows from \cite[Theorem 1]{DGHY} that $T_f^\omega$ is bounded on $A_\omega^p$. Moreover, by the proof of \cite[Theorem 1]{DGHY}, $\mu$ is a~$p$-Carleson measure for $A_\omega^p$, where $d\mu=|f| \omega dA$. Then by a similar proof as in Proposition \ref{TinWL}, we also have $T^\omega_{f}\in \mathrm{WL}_\omega^p$, and (i) is proved. Furthermore, (i) combined with Theorem \ref{main} yields the desired result in (ii).
\end{proof}

Recall that \cite[Theorem 2]{PRS} shows that a finite positive Borel measure $\mu$ is a $p$-Carleson measure for $A_\omega^p$ if and only if $\widetilde{T^\omega_{\mu}}$ is bounded on $\mathbb{D}$.
Now let $\mathrm{BT_\omega}$ be the collection of all measurable functions $f$ on $\mathbb{D}$ with $\widetilde{T^\omega_{|f|}}$ bounded.
From the proof of Proposition \ref{TBMO}, we know that $T^\omega_{f}\in \mathrm{WL}_\omega^p$ if $f\in \mathrm{BT_\omega}$.
\begin{corollary}\label{qcar}
Let $\omega\in\mathcal{D}$ and $1<p<\infty$. If $f\in\mathrm{BT_\omega}$, then $T_f^\omega$ is compact on $A_\omega^p$ if and only if $
\lim\limits_{|z|\to 1^-}\widetilde{T}(z)=0
$.
\end{corollary}

\section{Two-weight Bergman projection and Toeplitz operator}
In this section, we first prove the boundedness of the Bergman projection in the two-weight setting. Next, we establish a compact extrapolation theorem adapted to our setting later. Finally, we characterize the compactness of Toeplitz operators with bounded symbols in the two-weight setting.

Our approach to the two-weight Bergman projection employs the dyadic system in \cite{APR1,PRW},  which was earlier developed in \cite{Cm,GJ,Mei}.
Define the dyadic grids
$$
\mathcal{D}^\beta=\left\{I_{j, m}^\beta: j \in \mathbb{N} \cup\{0\},\quad m \in \mathbb{N} \cup\{0\},\quad 0 \leq m \leq 2^j-1\right\}, \quad \beta \in\left\{0,\frac{1}{3}\right\},
$$
where
$$
I_{j, m}^\beta=\left\{e^{i \theta}: \theta \in\left[\frac{4 \pi(m+\beta)}{2^j}, \frac{4 \pi(m+1+\beta)}{2^j}\right)\right\} .
$$
For an arc $I\subset \mathbb{T}$, define $$S(I)=\left\{r e^{i \theta}: e^{i \theta} \in I,\quad 1-|I| \leq r<1\right\}$$ is the Carleson square associated with $I$, and $|I|$ stands for the normalized arc-length of $I$.
For~$a \in \mathbb{D} \backslash\{0\}$, define $I_a=\{e^{i \theta}:|\arg (a e^{-i \theta})| \leq \frac{1-|a|}{2}\}$ and $S(a)=S(I_a).$ Now define the dyadic operator
\[
 Q_{\nu}^\beta f(z)
 =\sum_{I\in \mathcal D^\beta}
 \chi_{S(I)}(z)\frac{\int_{S(I)}f(\zeta)\nu(\zeta) dA(\zeta)}{\nu\big(S(I)\big)},\quad \beta \in\left\{0,\frac{1}{3}\right\}.
\]
The following lemma describes the relation between the maximal Bergman projection and the dyadic operators.
\begin{lemma}\label{PQ}
Let $\nu\in \mathcal{D}$. For any non-negative function $f\in L_{\nu,\mathrm{loc}}^1$,
\[
 P_{\nu}^+f(z)\lesssim\sum_{\beta\in\{0,\frac{1}{3}\}}Q_{\nu}^\beta f(z),\quad z\in\mathbb{D}.
\]
\end{lemma}
\begin{proof}
For every $z,\zeta\in\D$, there exists an arc $J\subset\T$ such that
$z,\zeta\in S(J)$ and
$|1-\bar{z}\zeta|\asymp |J|$, see \cite{APR1}.
Furthermore, there exist $\beta \in\left\{0,\frac{1}{3}\right\}$ and $L\in \mathcal{D}^\beta$ such that $J\subset L$ and $|L|\le 6|J|$, we refer to \cite{Mei} for more details.
Therefore, \eqref{kernel-} together with \eqref{Dhat2} and \eqref{Dhat1} yields
\begin{align*}
  |B_z^{\nu}(\zeta)|
 &\lesssim
   \frac{1}{\nu_{\frac{2}{|1-\bar{z}\zeta|}}|1-\bar{z}\zeta|}
  \asymp
 \frac{1}{\widehat{\nu}(1-|J|)|J|}
 \asymp
 \frac{1}{\widehat{\nu}(1-|L|)|L|}\\
 &\lesssim\sum_{\substack{I\in \mathcal{D}^\beta\\ L\subset I}}
 \frac{\chi_{S(I)}(z)\chi_{S(I)}(\zeta)}
 {\widehat{\nu}(1-|I|)|I|}
 \leq \sum_{\beta\in\{0,\frac{1}{3}\}}\sum_{I\in \mathcal{D}^\beta}
 \frac{\chi_{S(I)}(z)\chi_{S(I)}(\zeta)}
 {\widehat{\nu}(1-|I|)|I|},\quad{z,\zeta\in \D},
\end{align*}
which completes the proof.
\end{proof}
Recall that for a positive Borel measure $\mu$ and a dyadic grid $\mathcal{D}^\beta$ on $\mathbb{T}$, the dyadic weighted H\"{o}rmander type maximal function is defined by
$$
M_{\mu,\beta}(f)(z)=\sup _{\substack {I \in \mathcal{D}^\beta\\ z\in S(I)}} \frac{1}{\mu\big(S(I)\big)} \int_{S(I)}|f(\zeta)| \mu(\zeta)dA(\zeta),\quad z\in\mathbb{D}, \quad \beta\in\left\{0,\frac{1}{3}\right\}.
$$

\begin{theorem}\label{BQ}
Let $\nu\in \mathcal{D}$. If $\omega\in B_p(\nu$),
then
$$
\|Q_{\nu}^\beta f\|_{L^p_{\omega\nu} \rightarrow L^p_{\omega\nu}} dA\lesssim [\omega]_{B_p(\nu)}^{\max \{1,\frac{1}{p-1}\}}\|f\|_{L^p_{\omega\nu}},
\quad \beta\in\left\{0,\frac{1}{3}\right\}.
$$
\end{theorem}
\begin{proof}
We first prove the case $p=2$.
Note that $$\|Q_{\nu}^\beta\|_{L^2_{\omega\nu}\rightarrow L^2_{\omega\nu}}\asymp\|Q_{\nu}^\beta(\omega^{-1}\cdot)\|_{L^2_{\omega^{-1}\nu}
 \rightarrow L^2_{\omega\nu}}.$$
 For a non-negative function $f\in L^2_{\omega^{-1}\nu}$, by duality, we have
$$
 \|Q_{\nu}^\beta(\omega^{-1}f)\|_{L^2_{\omega\nu}}
 =\sup_{\substack{g\ge0\\ \|g\|_{L^2_{\omega\nu}}\le1}}
 \int_\D Q_{\nu}^\beta(\omega^{-1}f) g\omega\nu dA.
$$
\cite[Lemma A]{PRW} implies that both $M_{\omega\nu, \beta}: L_{\omega\nu}^2\rightarrow L_{\omega\nu}^2$ and $M_{\omega^{-1}\nu, \beta}: L_{\omega^{-1}\nu}^2\rightarrow L_{\omega^{-1}\nu}^2$ are bounded.
This combined with H\"older's inequality shows that
$$
\begin{aligned}
 &\int_\D Q_{\nu}^\beta(\omega^{-1}f) g \omega\nu dA\\
 &= \sum_{I\in\mathcal D^\beta}\nu\big(S(I)\big)
 \bigg(\frac{\int_{S(I)}f\omega^{-1}\nu dA}{\int_{S(I)}\omega^{-1}\nu dA}\bigg)
 \bigg(\frac{\int_{S(I)}g\omega\nu dA}{\int_{S(I)}\omega\nu dA}\bigg) \frac{\int_{S(I)}\omega\nu dA \int_{S(I)}\omega^{-1}\nu dA}{\nu\big(S(I)\big)^2}\\
 &\lesssim [\omega]_{B_2(\nu)}
 \sum_{I\in \mathcal{D}^\beta}\nu\big(T(I)\big)
  \bigg(\frac{\int_{S(I)}f\omega^{-1}\nu dA}{\int_{S(I)}\omega^{-1}\nu dA} \bigg)
  \bigg(\frac{\int_{S(I)}g\omega\nu dA}{\int_{S(I)}\omega\nu dA} \bigg)\\
 &\leq[\omega]_{B_2(\nu)}\int_\D
 \left(M_{\omega^{-1}\nu,\beta}(f) \omega^{-\frac{1}{2}}\right)
 \left(M_{\omega\nu,\beta}(g) \omega^{\frac{1}{2}}\right)\nu dA\\
 &\lesssim [\omega]_{B_2(\nu)}
 \|f\|_{L^2_{\omega^{-1}\nu}}
 \|g\|_{L^2_{\omega\nu}},
 \end{aligned}
$$
where $T(I)
=\{r e^{i \theta}: e^{i \theta} \in I,\quad 1-|I| \leqslant r<1-\frac{|I|}{2}\}.$
This estimate also holds when $2<p<\infty$ by an extrapolation argument similar to that used in \cite[Proposition 4.4]{PottRe}, while the case $1<p<2$ follows by duality.
\end{proof}
To prove the necessity of the condition $\omega \in B_{p}(\nu)$ for the boundedness of $P_\nu: L^p_{\omega\nu} \rightarrow L^{p, \infty}_{\omega\nu}$, we first provide the following local lower bound kernel estimate based on \cite[Lemma 7]{PRS} and its proof.
For $\nu\in\widehat{\mathcal{D}}$, there exists $\delta(\nu)\in (0,1)$ such that
\begin{equation}\label{BP1}
 \operatorname{Re}B_a^{\nu}(z)\gtrsim \frac{1}{\nu\big(S(a)\big)},\quad z\in S(a_\delta),\quad a\in \D,
\end{equation}
where $a_\delta=(1-\delta(1-|a|))e^{i\arg a}.$
Let $K$ be the constant appearing in \eqref{Ddef}. For an arc $|J|\leq \frac{\delta}{2K}$, set
$h=\frac{2K}{\delta}|J|,$ and we denote
\[
 T_K^J=\left\{re^{i\theta}:e^{i\theta}\in J,\quad 1-h\le r\le 1-\frac{h}{K}\right\}.
\]
Observe that $S(J)\subset S(a_\delta)$ when $a\in T_K^J$. Indeed, if $a\in T_K^J$, then $\delta(1-|a|)\ge 2|J|.$
For any $z\in S(J)$, we have $|z|\ge1-|J|\geq |a_\delta|.$
Moreover, since $e^{i\arg a}\in J$ and $\arg a_\delta=\arg a$, we have
$|\arg z-\arg a_\delta|\le |J|<\frac{1-|a_\delta|}{2}.$
It follows that $z\in S(a_\delta)$, and hence $S(J)\subset S(a_\delta)$.
This together with \eqref{BP1} shows that
\begin{equation}\label{BP3}
 \operatorname{Re}B_a^{\nu}(z)\gtrsim \frac1{\nu\big(S(a)\big)},\quad a\in T_K^J,\quad z\in S(J).
\end{equation}
Now we are ready to prove Theorem \ref{main1}.
\vspace{8pt}

\noindent \textbf{Proof of Theorem \ref{main1}.}
Assume (iii).
Let $K$ and $\delta$ be as in \eqref{Ddef} and \eqref{BP1}, respectively.
For each arc $J\subset \T$ with $|J|\leq \frac{\delta}{2K}$, the estimate \eqref{BP3} together with \eqref{Dhat1}  implies that there exists $C_1=C_1(\nu)>0$ such that
\begin{equation}\label{LP1}
 |P_{\nu}f(a)|\ge \operatorname{Re}P_{\nu}f(a)
 \geq  \frac{C_1}{\nu\big(S(J)\big)}\int_{S(J)}f\nu dA, \quad a\in T_K^J,
\end{equation} for all non-negative functions $f$ supported on $S(J)$.
Moreover, by \eqref{BP3} and \eqref{Ddef}, there exists $C_2=C_2(\nu)>0$ such that
\begin{equation}\label{LP2}
 |P_{\nu}g(a)|
 \ge \operatorname{Re}P_{\nu}g(a)
 \geq  \frac {C_2}{\nu(T_K^J)}\int_{T_K^J} g\nu dA, \quad a\in S(J),
\end{equation}
for all non-negative functions $g$ supported on $T_K^J$. Moreover, there exists $C=C(\omega,\nu,p)>0$ such that
\[
 \lambda^p \omega\nu\big(\{a\in\mathbb{D}:|P_{\nu}f(a)|>\lambda\}\big)
 \le C\|f\|_{L_{\omega\nu}^p}^p,
\]
for each $\lambda>0$.
Choose
\[
 \lambda_1=\frac{C_1}{\nu\big(S(J)\big)}\int_{S(J)}
 \min\{n,\omega^{-\frac{p'}{p}} \}\nu dA,
 \qquad
 f=\chi_{S(J)}\min\{n,\omega^{-\frac{p'}{p}} \}, \quad n\in \mathbb{N}.
\]Furthermore, \eqref{LP1} yields
\[
 T_K^J\subset\left\{a\in\mathbb{D}:|P_{\nu}f(a)|\geq C_1
 \frac{\int_{S(J)}f\nu dA}{\nu\big(S(J)\big)}\right\}.
\]
It follows that
\[
 \frac{\left(\int_{S(J)}
 \min\{n,\omega^{-\frac{p'}{p}} \}\nu dA\right)^{p-1}(\omega\nu)(T_K^J)}{\nu\big(S(J)\big)^p}\,
 \leq\frac{C}{C_1^p}.
\]
Similarly, choose
\[
 \lambda_2=\frac{C_2}{\nu(T_K^J)}\int_{T_K^J}
 \min\{n,\omega^{-\frac{p'}{p}} \}\nu dA,
 \qquad
 g=\chi_{T_K^J}\min\{n,\omega^{-\frac{p'}{p}} \},\quad n\in \mathbb{N}.
\]
By \eqref{LP2}, we have
\[
 S(J)\subset\left\{a\in\D:|P_{\nu}g(a)|\geq C_2
 \frac{\int_{T_K^J}g\nu dA}{\nu(T_K^J)}\right\}.
\]
Then
\[
 \frac{\left(\int_{T_K^J}
 \min\{n,\omega^{-\frac{p'}{p}} \}\nu dA\right)^{p-1}(\omega\nu)\big(S(J)\big)}{\nu(T_K^J)^p}
 \leq \frac{C}{C_2^{p}}.
\]
Hence $\omega\in L_\nu^1$.
Letting $n\to\infty$ and applying Fatou's lemma gives
\[
 \frac{\left(\int_{S(J)}\omega^{-\frac{p'}{p}} \nu dA\right)^{p-1}(\omega\nu)(T_K^J)}{\nu\big(S(J)\big)^p}
 \frac{\left(\int_{T_K^J}\omega^{-\frac{p'}{p}} \nu dA\right)^{p-1}(\omega\nu)\big(S(J)\big)}{\nu(T_K^J)^p}
 \leq \frac{C^2}{C_1^{p}C_2^{p}}.
\]
By H\"older's inequality, we obtain
\[
 \frac{(\omega\nu)(T_K^J)\left(\int_{T_K^J}\omega^{-\frac{p'}{p}} \nu dA\right)^{p-1}}{\nu(T_K^J)^p}\ge1.
\]
Therefore
\[
 \frac{(\omega\nu)\big(S(J)\big) \left(\int_{S(J)}\omega^{-\frac{p'}{p}} \nu dA\right)^{p-1}}{\nu\big(S(J)\big)^p}
\]
is uniform bounded over all Carleson squares $S(J)$ with $|J|\leq \frac{\delta}{2K}$. Thus $\omega\in B_p(\nu)$, which yields (iv).
The implication (i) $\Rightarrow$ (ii) $\Rightarrow$ (iii) is straightforward, while (iv) $\Rightarrow$ (i) follows from Lemma \ref{PQ} and Theorem \ref{BQ}.
\qed

\vspace{8pt}
Now we proceed to consider the compactness of Toeplitz operators with bounded symbols in the two-weight setting.
To this end, we recall several classes of non-radial weights, see \cite{KR,PR} for more details.
We write $\omega \in \widehat{\mathcal{D}}(\mathbb{D})$ if there exists $C=C(\omega)>0$ such that
$$
\omega\big(S(a)\big) \leqslant C \omega\left(S\bigg(\frac{1+|a|}{2} e^{i \arg a}\bigg)\right), \quad a \in \mathbb{D} \backslash\{0\} .
$$
The $K$-top of a Carleson box $S(a)$ is
$$
T_K(a)=\bigg\{r e^{i t}: e^{i t} \in I_a,\quad |a| \leqslant r<1-\frac{1-|a|}{K}\bigg\}.
$$
For an arc $I\subset \mathbb{T}$, we write $T_K(I)=\{r e^{i t}: e^{i t} \in I,\quad 1-|I| \leqslant r<1-\frac{|I|}{K}\}$.
Further, a weight~$\omega$ on $\mathbb{D}$ belongs to $\widecheck{\mathcal{D}}(\mathbb{D})$ if there exist $K=K(\omega)>1$ and $C=C(\omega)>1$ such that
$$
\omega\big(S(a)\big) \leqslant C \omega\big(T_K (a)\big), \quad a \in \mathbb{D} \backslash\{0\} .
$$
Denote $\mathcal{D}(\mathbb{D})=\widehat{\mathcal{D}}(\mathbb{D})\cap\widecheck{\mathcal{D}}(\mathbb{D}).$

For a weight $\nu$, let
$$B_\infty(\nu)=\bigcup\limits_{1<p<\infty} \{\omega: \omega\nu^{-1} \in B_p(\nu)\}.$$
We say that a weight $\omega$ has Kerman-Torchinsky $KT(\nu)$-property if there exist constants $\delta=\delta(\nu)\in(0,1)$ and $C=C(\nu)>0$ such that
$$\label{KT}
\frac{\nu(E)}{\nu(S)} \leq C \bigg( \frac{\omega(E)}{\omega(S)} \bigg)^\delta
$$
for all Carleson squares $S$ and measurable sets $E\subset S$.
We recall the following result, see \cite[Theorem 3.1(c)]{DMO} or \cite[Proposition A]{PR}.
\begin{proposition}\label{KTnu}
Let $\nu$ be a weight on $\mathbb{D}$. Then a weight $\omega$ belongs to $B_\infty(\nu)$ if and only if it has $KT(\nu)$-property.
\end{proposition}
By Proposition \ref{KTnu}, it follows that $$\bigcup\limits_{\nu\in \widehat{\mathcal{D}}} B_\infty(\nu)\subset \widehat{\mathcal{D}}(\mathbb{D}) \quad \mathrm{and} \quad \bigcup\limits_{\nu\in \mathcal{D}}B_\infty(\nu)\subset \mathcal{D}(\mathbb{D}).$$
For a weight $\nu$, set $\vartheta=\vartheta_{\omega,\nu}=\big(\omega\nu^{-1}\big)^{-\frac{p'}{p}}\nu=\omega^{-\frac{p'}{p}}\nu^{\frac{p}{p-1}}$. If $\omega\nu^{-1}\in B_p(\nu)$, then by \eqref{Bp}, we have
$$
\sup_S \frac{\omega(S)}{\nu(S)} \bigg(\frac{\vartheta(S) }{\nu(S)}\bigg)^{p-1}<\infty,
$$
from which we know that $\omega\nu^{-1}\in B_p(\nu)$ is equivalent to $\vartheta\nu^{-1}\in B_{p'}(\nu)$. Furthermore, if $\omega\nu^{-1}\in B_p(\nu)$, then for any $f \in L_\omega^p$, H\"{o}lder's inequality yields
$$
\|f\|_{L_\nu^1} \leq \|f\|_{L_\omega^p} \left( \int_{\mathbb{D}} \omega^{-\frac{p'}{p}} \nu^{p'}  dA \right)^{\frac{1}{p'}} = \|f\|_{L_\omega^p} \, \vartheta(\mathbb{D})^{\frac{1}{p'}},
$$
which implies $L_\omega^p\subset L_\nu^1$.

\cite[Chapter 2]{Zhu2} provides an elementary introduction to the basic theory of complex interpolation, including the interpolation of
$L^p$ spaces. Taking this theory as known, next we prove a complex interpolation theorem on the weighted Bergman spaces, following the overall structure of the proof in \cite[Theorem 5.1]{AP}.
\begin{proposition}\label{inter1}
Let $\nu\in \mathcal{D}$, $\omega_0$, $\omega_1$ be weights, $p_0,p_1\in(1,\infty)$ and $\theta \in(0,1)$. If $\omega_0\nu^{-1}\in B_{p_0}(\nu)$ and $\omega_1\nu^{-1}\in B_{p_1}(\nu)$, then
$$
[A_{\omega_0}^{p_0},A_{\omega_1}^{p_1}]_{\theta}=A_{\omega}^p
$$
with equivalent norms, where
$$
\frac{1}{p}=\frac{\theta}{p_0}+\frac{1-\theta}{p_1} \quad and \quad \omega^{\frac{1}{p}}=\omega_0^{\frac{\theta}{p_0}} \omega_1^{\frac{1-\theta}{p_1}}.
$$
\end{proposition}
\begin{proof}
We first show $\omega\nu^{-1} \in B_p(\nu)$. Since $\frac{p\theta}{p_0}+\frac{p(1-\theta)}{p_1}=1$ and $\frac{p'\theta}{p_0'}+\frac{p'(1-\theta)}{p_1'}=1$, by H\"{o}lder's inequality, we obtain
$$
\begin{aligned}
&\frac{\omega(S)}{\nu(S)}\bigg( \frac{\int_{S}(\omega\nu^{-1})^{-\frac{p'}{p}}\nu dA}{\nu(S)} \bigg)^{p-1}\\
=&\frac{\int_{S}\big(\omega_0\nu^{-1}\big)^{\frac{p\theta}{p_0}} \big(\omega_1\nu^{-1}\big)^{\frac{p(1-\theta)}{p_1}}\nu dA}{\nu(S)}
\bigg(\frac{\int_{S}\big((\omega_0\nu^{-1})^{-\frac{p_0'}{p_0}}\big)^{\frac{p'\theta}{p_0'}} \big((\omega_1\nu^{-1})^{-\frac{p_1'}{p_1}}\big)^{\frac{p'(1-\theta)}{p_1'}}\nu dA}{\nu(S)} \bigg)^{p-1}\\
\leq& \bigg(\frac{\omega_0(S)}{\nu(S)}\bigg)^{\frac{p\theta}{p_0}}\bigg(\frac{\omega_1(S)}{\nu(S)}\bigg)^{\frac{p(1-\theta)}{p_1}}
\bigg( \frac{\int_{S}(\omega_0\nu^{-1})^{-\frac{p_0'}{p_0}}\nu dA}{\nu(S)}\bigg)^{\frac{p\theta}{p_0'}}\bigg( \frac{\int_{S}(\omega_1\nu^{-1})^{-\frac{p_1'}{p_1}}\nu dA}{\nu(S)}\bigg)^{\frac{p(1-\theta)}{p_1'}}\\
\leq& [\omega_0\nu^{-1}]_{B_{p_0}(\nu)}^{\frac{p\theta}{p_0}} [\omega_1\nu^{-1}]_{B_{p_1}(\nu)}^{\frac{p(1-\theta)}{p_1}},
\end{aligned}
$$ for every Carleson square $S$, which implies $\omega\nu^{-1} \in B_p(\nu)$.

Let $\mathcal{X}_\theta = \left[A_{\omega_0}^{p_0}, A_{\omega_1}^{p_1}\right]_\theta$. The inclusion $\mathcal{X}_\theta \subseteq A_\omega^p$ follows from the isometric embedding $A_{\omega_j}^{p_j} \rightarrow L_{\omega_j}^{p_j}$ for $j = 0,1$, the fundamental properties of complex interpolation, and the classical Stein-Weiss interpolation theorem (see \cite{BL}), which asserts that
$$
\big[L_{\omega_0}^{p_0}, L_{\omega_1}^{p_1}\big]_\theta = L_\omega^p
$$
with equal norms. In particular, for any $f \in \mathcal{X}_\theta$, we have
$$
\|f\|_{L_\omega^p} = \|f\|_{A_\omega^p} \leq \|f\|_{\mathcal{X}_\theta}.
$$
Next we show that the reverse inclusion holds.
For $f\in A_\omega^p$ and $\|f\|_{A_\omega^p}\leq 1$, by Stein-Weiss interpolation theorem, there exists a function $F_\zeta(z)$, defined for $z \in \mathbb{D}$ and $\zeta$ in the closed strip $\{ \zeta \in \mathbb{C} : 0 \leq \operatorname{Re} \zeta \leq 1 \}$, such that
\begin{itemize}
\item[(i)] $F_\theta(z) = f(z)$ for all $z \in \mathbb{D}$;
\item[(ii)] $\|F_\zeta\|_{L_{\omega_0}^{p_0}} \lesssim 1$ for all $\zeta$ with $\mathrm{Re} \zeta = 0$;
\item[(iii)] $\|F_\zeta\|_{L_{\omega_1}^{p_1}} \lesssim 1$ for all $\zeta$ with $\mathrm{Re} \zeta = 1$.
\end{itemize}
Let $G_\zeta(z)=P_\nu F_\zeta(z)$. It follows from Theorem \ref{main1} that the Bergman projection $P_\nu:L_\omega^p\rightarrow L_\omega^p$ is bounded. Then, by the boundedness of the Bergman projection $P_\nu:L_{\omega_i}^{p_i}\rightarrow L_{\omega_i}^{p_i}, i=0,1$,~$G_\zeta$ defines an analytic function from the strip $\{ \zeta \in \mathbb{C} : 0 \leq \operatorname{Re} \zeta \leq 1 \}$ to $A_{\omega_0}^{p_0}+ A_{\omega_1}^{p_1}$. Further, the containment $L_\omega^p\subset L_\nu^1$ yields
\begin{itemize}
\item[(i)] $G_\theta(z)=P_\nu (f)(z)=f(z)$ for all $z \in \mathbb{D}$;
\item[(ii)] $\left\|G_\zeta\right\|_{A^{p_0}_{\omega_0}} \leq \|P_\nu\|_{L_{\omega_0}^{p_0}\rightarrow L_{\omega_0}^{p_0}}\|F_\zeta\|_{L_{\omega_0}^{p_0}}\lesssim 1$ for all $\zeta$ with $\mathrm{Re} \zeta=0$;
\item[(iii)] $\left\|G_\zeta\right\|_{A^{p_1}_{\omega_1}} \leq \|P_\nu\|_{L_{\omega_1}^{p_1}\rightarrow L_{\omega_1}^{p_1}}\|F_\zeta\|_{L_{\omega_1}^{p_1}}\lesssim 1$ for all $\zeta$ with $\mathrm{Re} \zeta=1$.
\end{itemize}
Therefore, we conclude that $f$ belongs to $\mathcal{X}_\theta$, and $\|f\|_{\mathcal{X}_\theta}\lesssim \|f\|_{A_\omega^p}$. This completes the proof.
\end{proof}
For $q>1$ and a weight $\nu$, we say that a weight $\omega$ belongs to the reverse H\"{o}lder class $RH_{q,\nu}$ if there exists a constant $C=C(\nu,q)>0$ such that
$$
\bigg(\frac{1}{\nu(S)}\int_{S}\omega^q\nu dA\bigg)^{\frac{1}{q}}\leq\frac{C}{\nu(S)}\int_{S}\omega \nu dA,
$$
for every Carleson square $S$.
The following result is an analogue of \cite[Lemma 3.9]{SW}. Furthermore, the methodology for the estimates presented here is adapted from the approach outlined in Lemma 4.4 of \cite{HL}.
\begin{lemma}\label{inter2}
Let $\nu\in \mathcal{D}$ and $p,p_0,q\in(1,\infty)$. If weights $\omega$, $\omega_0$ satisfy $\omega\nu^{-1}\in B_p(\nu) \cap RH_{q,\nu}$ and $\omega_0\nu^{-1}\in B_{p_0}(\nu) \cap RH_{q,\nu}$ with $(\omega\nu^{-1})^{-\frac{p'}{p}},(\omega_0\nu^{-1})^{-\frac{p_0'}{p_0}}\in RH_{q,\nu}$, then there exist $p_1\in(1,\infty)$, a weight $\omega_1$ with $\omega_1 \nu^{-1}\in B_{p_1}(\nu)$, and $\theta\in (0,1)$ such that
$$
A_\omega^p=[A_{\omega_0}^{p_0},A_{\omega_1}^{p_1}]_\theta.
$$
\end{lemma}
\begin{proof}
By Proposition \ref{inter1}, for $p_1$ and $\omega_1$, we now show that there exists $\theta\in(0,1)$ such that
$$
p_1(\theta)= \frac{1-\theta}{\frac{1}{p}-\frac{\theta}{p_0}}\in(1,\infty)
$$
and the weight
$$
\omega_1(\theta)=\omega^{\frac{p_1}{p(1-\theta)}}\omega_0^{-\frac{p_1\theta}{p_0(1-\theta)}}
$$
satisfies $\omega_1\nu^{-1}\in B_{p_1}(\nu)$.
Since $p_1(\theta)$ is continuous at $0$ and $p_1(0)=p\in (1,\infty)$, we have $p_1(\theta)\in(1,\infty)$ for sufficiently small $\theta>0$.
To show that $\omega_1\nu^{-1}\in B_{p_1}(\nu)$, we estimate
$$
\frac{\omega_1(S)}{\nu(S)}\bigg(\frac{\int_{S}\big(\frac{\omega_1}{\nu}\big)^{-\frac{p_1'}{p_1}}\nu dA}{\nu(S)}\bigg)^{p_1-1},
$$ for all Carleson square $S$.
Since $\frac{p_1}{p(1-\theta)}=\frac{p_1\theta}{p_0(1-\theta)}+1$, by H\"{o}lder's inequality, for $\epsilon>0$ we have
$$
\begin{aligned}
\frac{\omega_1(S)}{\nu(S)}
=&\frac{\int_{S} \big(\frac{\omega}{\nu}\big)^{\frac{p_1}{p(1-\theta)}} \big(\frac{\omega_0}{\nu}\big)^{-\frac{p_1\theta}{p_0(1-\theta)}}\nu dA}{\nu(S)}\\
\leq &  \bigg(\frac{\int_{S} \big(\frac{\omega}{\nu}\big)^{\frac{p_1(1+\epsilon)}{p(1-\theta)}}\nu dA}{\nu(S)}\bigg)^{\frac{1}{1+\epsilon}}
\bigg( \frac{\big(\big(\frac{\omega_0}{\nu}\big)^{-\frac{p_0'}{p_0}}\big)^{\frac{p_1\theta(1+\epsilon)}{p_0'(1-\theta)\epsilon}}\nu dA}{\nu(S)} \bigg)^{\frac{\epsilon}{1+\epsilon}}.
\end{aligned}
$$
Moreover, notice that $\frac{p_1'\theta}{p_0(1-\theta)}-\frac{p_1'}{p(1-\theta)}=-\frac{1}{p_1-1}$. Then for $\delta>0$, H\"{o}lder's inequality gives
$$
\begin{aligned}
\frac{\int_{S}\big(\frac{\omega_1}{\nu}\big)^{-\frac{p_1'}{p_1}}\nu dA}{\nu(S)}
=&\frac{\int_{S} \big(\frac{\omega}{\nu}\big)^{-\frac{p_1'}{p(1-\theta)}}\big(\frac{\omega_0}{\nu}\big)^{\frac{p_1'\theta}{p_0(1-\theta)}} \nu dA}{\nu(S)}\\
\leq& \bigg(\frac{\int_{S}(\big(\frac{\omega}{\nu}\big)^{-\frac{p'}{p}}\big)^{\frac{p_1'(1+\delta)}{p'(1-\theta)}}\nu dA}{\nu(S)} \bigg)^{\frac{1}{1+\delta}}
\bigg(\frac{\int_{S}\big(\frac{\omega_0}{\nu}\big)^{\frac{p_1'\theta(1+\delta)}{p_0(1-\theta)\delta}} \nu dA }{\nu(S)}\bigg)^{\frac{\delta}{1+\delta}}.
\end{aligned}
$$
Choose $\epsilon= \frac{\theta p}{p_0'}$ and $\delta=\frac{\theta p'}{p_0}$. Combining the above two estimates, we have
$$\label{irh}
\begin{aligned}
\frac{\omega_1(S)}{\nu(S)}\bigg(\frac{\int_{S}\big(\frac{\omega_1}{\nu}\big)^{-\frac{p_1'}{p_1}}\nu dA}{\nu(S)}\bigg)^{p_1-1}
\leq &\bigg(\frac{\int_{S} \big(\frac{\omega}{\nu}\big)^{s(\theta)}\nu dA}{\nu(S)}\bigg)^{\frac{1}{1+\epsilon}}
\bigg( \frac{\int_{S}\big(\big(\frac{\omega_0}{\nu}\big)^{-\frac{p_0'}{p_0}}\big)^{s(\theta)}\nu dA}{\nu(S)} \bigg)^{\frac{\epsilon}{1+\epsilon}}\nonumber\\
 &\cdot \bigg(\frac{\int_{S}\big(\big(\frac{\omega}{\nu}\big)^{-\frac{p'}{p}}\big)^{t(\theta)}\nu dA}{\nu(S)} \bigg)^{\frac{p_1-1}{1+\delta}}
\bigg(\frac{\int_{S}\big(\frac{\omega_0}{\nu}\big)^{t(\theta)} \nu dA }{\nu(S)}\bigg)^{\frac{\delta(p_1-1)}{1+\delta}},
\end{aligned}
$$
where
$s(\theta)=\frac{p_1(\theta)(p_0'+\theta p)}{pp_0'(1-\theta)}$ and $ t(\theta)=\frac{p_1(\theta)'(p_0+\theta p')}{p'p_0(1-\theta)}$.
Since $s(\theta)$ and $t(\theta)$ are continuous at $0$ and $s(0)=t(0)=1$, we may choose $\theta$ small enough such that $\max(s(\theta),t(\theta))\leq q$.
Then, by the reverse H\"{o}lder assumption on weights $\omega\nu^{-1}$, $\omega_0\nu^{-1}$, $(\omega\nu^{-1})^{-\frac{p'}{p}}$ and $(\omega_0\nu^{-1})^{-\frac{p_0'}{p_0}}$, we obtain
\begin{align*}
&\frac{\omega_1(S)}{\nu(S)}\bigg(\frac{\int_{S}\big(\frac{\omega_1}{\nu}\big)^{-\frac{p_1'}{p_1}}\nu dA}{\nu(S)}\bigg)^{p_1-1}\\
\lesssim&\bigg(\frac{\omega(S)}{\nu(S)}\bigg)^{\frac{p_1}{p(1-\theta)}}\bigg(\frac{\int_{S}\big(\frac{\omega}{\nu}\big)^{-\frac{p'}{p}}\nu dA}{\nu(S)} \bigg)^{\frac{p_1}{p'(1-\theta)}}\bigg(\frac{\omega_0(S) }{\nu(S)}\bigg)^{\frac{p_1\theta}{p_0(1-\theta)}}\bigg( \frac{\int_{S}\big(\frac{\omega_0}{\nu}\big)^{-\frac{p_0'}{p_0}}\nu dA}{\nu(S)} \bigg)^{\frac{p_1\theta}{p_0'(1-\theta)}}\\
\leq& [\omega\nu^{-1}]_{B_p(\nu)}^{\frac{p_1}{p(1-\theta)}}[\omega_0\nu^{-1}]_{B_{p_0}(\nu)}^{\frac{p_1\theta}{p_0(1-\theta)}}.
\end{align*}
Taking supremum over all Carleson squares, we complete the proof.
\end{proof}
By using the main method in \cite{HL} and \cite[Theorem 1.7]{SW}, we establish the following compact extrapolation result.
\begin{theorem}\label{ce}
Let $T$ be a linear operator, $\nu\in\mathcal{D}$ and $1<p_0,q<\infty$. If $T$ is bounded on $A_\omega^{p_0}$ for all $\omega$ with $\omega\nu^{-1} \in B_{p_0}(\nu)$ and $T$ is compact on $A_{\omega_0}^{p_0}$ for some $\omega_0$ satisfying $\omega_0 \in B_{p_0}(\nu) \cap R H_{q,\nu}$ and $\big(\omega_0\nu^{-1})^{-\frac{p_0'}{p_0}} \in R H_{q,\nu}$, then $T$ is compact on $A_\omega^p$ for all $p \in(1, \infty)$ and all $\omega$ satisfying $\omega\nu^{-1} \in B_p(\nu) \cap R H_{q,\nu}$ and $\big(\omega\nu^{-1}\big)^{-\frac{p'}{p}}\in RH_{q,\nu}$.
\end{theorem}
\begin{proof}
Assume the weight $\omega_0$ satisfies $\omega_0 \in B_{p_0}(\nu) \cap R H_{q,\nu}$ and $\big(\omega_0\nu^{-1})^{-\frac{p_0'}{p_0}} \in R H_{q,\nu}$. Then for each $p \in(1, \infty)$ and every weight $\omega$ satisfies $\omega\nu^{-1} \in B_p(\nu) \cap R H_{q,\nu}$ with $\big(\omega\nu^{-1}\big)^{-\frac{p'}{p}}\in RH_{q,\nu}$, Lemma \ref{inter2} yields there exist $1<p_1<\infty$, a weight $\omega_1$ with $\omega_1\nu^{-1}\in B_{p_1}(\nu)$, and $\theta\in(0,1)$ such that
$$
A_\omega^p=[A_{\omega_0}^{p_0},A_{\omega_1}^{p_1}]_\theta.
$$
By the standard extrapolation result, $T$ is bounded on $A_{\omega_1}^{p_1}.$ Furthermore, \cite[Corollary 7]{CK} yields that $T_f^\nu$ is compact on $A_{\omega}^p$, which completes the proof.
\end{proof}
Based on Lemma \ref{inter2} and Theorem \ref{ce}, we obtain the following characterization of the compactness of Toeplitz operators with bounded symbols in the two-weight setting under certain reverse H\"{o}lder condition.
\begin{corollary}\label{BpRH}
Let $\nu \in \mathcal{D}$ and $1<p,q<\infty$. For $f\in L^\infty$, if a weight $\omega$ satisfies $\omega\nu^{-1}\in B_{p}(\nu)\cap RH_{q,\nu}$ and $\big(\omega\nu^{-1}\big)^{-\frac{p'}{p}}\in RH_{q,\nu}$, then the following statements are equivalent:
\begin{itemize}
\item [(i)] $T_f^\nu$ is compact on $A_\omega^p$;
\item[(ii)]$T_f^\nu$ is compact on $A_\nu^p$;
\item[(iii)]
$\lim\limits_{|z|\to 1^-} \widetilde{T}(z)=0.
$
\end{itemize}
\end{corollary}

\begin{proof}
(i) $\Rightarrow$ (ii). Assume that $T_f^\nu$ is compact on $A_\omega^p$, where the weight $\omega$ satisfies $\omega\nu^{-1}\in B_{p}(\nu)\cap RH_{q,\nu}$ with $(\omega\nu^{-1})^{-\frac{p'}{p}}\in RH_{q,\nu}$. It follows from Theorem \ref{main1} that $T_f^\nu$ is bounded on~$A_{\omega}^{p}$.
This combined with Theorem \ref{ce} yields that $T_f^\nu$ is compact on $A_\nu^p$.

The implication (ii) $\Rightarrow$ (i) follows from Theorem \ref{main1} and Theorem \ref{ce} similarly. The equivalence (ii) $\Leftrightarrow$ (iii) is a consequence of Theorem \ref{main} and Proposition \ref{TinWL}, and therefore we complete the proof.
\end{proof}
We recall the following useful property of weights in $\widehat{\mathcal{D}}(\mathbb{D})$ \cite{KR}.
For a weight $\omega$ on $\mathbb{D}$ such that $\omega\big(S(a)\big)>0$ for all $a\in \mathbb{D}\setminus\{0\}$, then $\omega\in \widehat{\mathcal{D}}(\mathbb{D})$ if and only if there exist $\beta=\beta(\omega)>0$ and $C=C(\omega)\geq 1$ such that
\begin{equation}\label{DhatD}
\frac{\omega\big(S(a)\big)}{(1-|a|)^\beta}\leq C \frac{\omega\big(S(a')\big)}{(1-|a'|)^\beta}, \quad 0<|a|\leq |a'|<1,\quad \arg a=\arg a'.
\end{equation}
A weight $\omega$ is called essentially (or almost) constant in each hyperbolically bounded region, if for any $s>0$, there exists $C(\omega,s)>1$ such that
\begin{equation}\label{ec}
C^{-1}\omega(\zeta)\leq \omega(z)\leq C\omega(\zeta), \quad{\beta(z,\zeta)<s}, \quad z,\zeta\in \D.
\end{equation}
For $\omega\in \mathcal{D}(\D)$ and $\nu\in\mathcal{D}$, there exist $r_1=r_1(\omega)>0$ and $r_2=r_2(\nu)>0$ such that both $\omega\big(D(\cdot,r)\big)$ and $\nu\big(D(\cdot,r)\big)$ satisfy \eqref{ec}, for each $ r\geq\max\{r_1,r_2\}.$ From now on, we denote $ R=R(\omega,\nu)=\max\{r_1,r_2\}$.
Let $$\sigma_{\omega,\nu ,r}(z)=\frac{\omega\big(D(z,r)\big)}{\nu \big(D(z,r)\big)}\nu(z),\quad z\in\mathbb{D},$$ where $r\geq R$. Next we prove that $A_{\sigma_{\omega,\nu ,r}}^p=A_\omega^p$ with an equivalent norm.
\begin{lemma}\label{normeq}
Let $\omega\in \mathcal{D}(\D)$ and $\nu\in\mathcal{D}$. Then $\sigma_{\omega,\nu ,r}(S)\asymp\omega(S)$ uniformly over all Carleson squares $S$ for every $r\geq R$. Furthermore,
\begin{equation}\label{nmeq1}
\|f\|_{A^p_{\sigma_{\omega,\nu ,r}}}\asymp\|f\|_{A^p_\omega}, \quad f\in H(\D).
\end{equation}
\end{lemma}
\begin{proof}
Fubini's theorem together with \eqref{DhatD} shows that
$$
\begin{aligned}
\sigma_{\omega,\nu ,r}\big(S(a)\big)
=&\int_{S(a)}\left(\int_{D(\zeta,r)}\omega(z)dA(z)\right)\frac{\nu(\zeta)}{\nu\big(D(z,r)\big)}dA(\zeta)\\
=&\int_{{\{z\in\mathbb{D}:S(a)\cap D(z,r)\neq\o \}}}\left(\int_{D(z,r)\cap S(a)}\frac{\nu(\zeta)}{\nu\big(D(z,r)\big)}dA(\zeta)\right)\omega(z)dA(z)\\
\lesssim &\omega\big(S(b)\big) \lesssim \omega\big(S(a)\big), \quad |a|>\tanh r
\end{aligned}
$$
where $b\in\mathbb{D}$ is chosen so that $\arg b=\arg a$ and $1-|b|\asymp 1-|a|$.

On the other hand, for $\nu\in \mathcal{D}$ and every Carleson square $S$, we have
\begin{equation}\label{eq:intersection}
 \nu\bigl(S\cap D(\zeta,r)\bigr)\gtrsim\nu\bigl(D(\zeta,r)\bigr),\quad \zeta\in S.
\end{equation}
Indeed, write $S=S(I)$ and fix $\zeta=\rho e^{i\theta}\in S$.
Since the center $e^{i\theta}$ of $J_\zeta$ belongs to $I$ and $|J_\zeta|\le |I|$, we have
$
 |I\cap J_\zeta|\ge\frac12|J_\zeta|,
$
where $J_\zeta$ is the boundary arc of $S(\zeta)$.
Since $\nu\in\mathcal{D}$, it follows that
\[
\nu\bigl(S\cap D(\zeta,r)\bigr)
 \ge \nu\bigl(S\cap T_K(\zeta)\bigr)
 =\frac{|I\cap J_\zeta|}{|J_\zeta|}\,\nu\big(T_K(\zeta)\big)
 \ge\frac12\nu \big(T_K(\zeta)\big)
 \gtrsim \nu \big(S(\zeta)\big),
\]
which proves \eqref{eq:intersection}. This combined with Fubini's theorem yields
$$
 \sigma_{\omega,\nu ,r}(S)
 \ge\int_S\omega(\zeta)
 \left(\int_{S\cap\Delta(\zeta,r)}
 \frac{\nu(z)}{\nu\big(D(z,r)\big)}dA(z)\right)dA(\zeta)
 \gtrsim \omega(S).
$$
 Then \eqref{nmeq1} follows from \cite[Theorem 1.3]{PR}.
\end{proof}
We now adapt the machinery from \cite{APR2} to derive the next lemma.
\begin{lemma}\label{invarRH}
Let $\nu\in\mathcal D$ and $1<p<\infty$. Suppose $\eta\in B_p(\nu)$ and that $\eta$ is essentially constant on hyperbolically bounded regions. Then there exists $q=q(\nu)>1$ such that $\eta\in RH_{q,\nu}$.
\end{lemma}

\begin{proof}

Fix a Carleson square $S(I)$ and $N>1$. Given an arc $I\subset\mathbb T$, we denote by $\mathcal D_{I}$
the collection of all dyadic descendants of $I$ obtained by repeated
bisection. More precisely, the first generation consists of the two
disjoint subarcs of $I$ of length $|I|/2$, and subsequent generations
are defined recursively by bisecting each dyadic subarc. The stopping children~$\mathcal L(I)$ of a given $I$ is the family of maximal dyadic arcs $L\in \mathcal{D}_{I}$ satisfying
$$\label{eq:stopping}
 \frac{(\eta\nu)\big(S(L)\big)}{\nu\big(S(L)\big)}
 >N\frac{(\eta\nu)\big(S(I)\big)}{\nu\big(S(I)\big)}.
$$
Let $\mathcal L_1=\mathcal L(I).$
For each $k\ge2$, let
$ \mathcal L_{k}=\bigcup\limits_{L\in\mathcal L_{k-1}(I)}\mathcal L(L).$ Hence $\mathcal L_k$ is the $k$-th stopping generation associated with the initial $I$.
For $L\in \mathcal L_k$, since the members of $\mathcal L_{k+1}$ are pairwise disjoint, we have
$$\label{eq:numass}
 \sum_{\substack{L'\in\mathcal L_{k+1}\\ L'\subset L}}\nu\big(S(L')\big)
 <N^{-1}\frac{\nu\big(S(L)\big)}{(\eta\nu)\big(S(L)\big)}\sum\limits_{\substack{L'\in\mathcal L_{k+1}\\ L'\subset L}}(\eta\nu)\big(S(L')\big)
 \le N^{-1}\nu\big(S(L)\big).
$$
It follows that
\begin{equation}\label{eq:Fnumass}
\nu\bigg(S(L)\setminus\bigcup_{\substack{L'\in\mathcal L_{k+1}\\ L'\subset L}}S(L')\bigg)
 \ge(1-N^{-1})\nu\big(S(L)\big).
\end{equation}
On the other hand, H\"older's inequality gives
$$
\begin{aligned}
&\nu\bigg(S(L)\setminus\bigcup_{\substack{L'\in\mathcal L_{k+1}\\ L'\subset L}}S(L')\bigg) \nu\big(S(L)\big)^{-1}\\
 \le&
\left((\eta\nu)\bigg(S(L)\setminus\bigcup_{\substack{L'\in\mathcal L_{k+1}\\ L'\subset L}}S(L')\bigg)\right)^{1/p}
 \bigg(\int_{S(L)} \eta^{-\frac{p'}{p}}\nu dA\bigg)^{1/p'}{\nu\big(S(L)\big)}^{-1}\\
  \leq& [\eta]_{B_p(\nu)}^{1/p}
 \left(
(\eta\nu)\bigg(S(L)\setminus\bigcup_{\substack{L'\in\mathcal L_{k+1}\\ L'\subset L}}S(L')\bigg)
 \bigg({(\eta\nu)\big(S(L)\big)}\bigg)^{-1}
 \right)^{1/p}.
 \end{aligned}
$$
This combined with \eqref{eq:Fnumass} shows that there exists $c_0=c_0(\eta,\nu,p)\in (0,1)$ such that
\[
(\eta\nu)\bigg(\bigcup_{\substack{L'\in\mathcal L_{k+1}\\ L'\subset L}}S(L')\bigg)= (\eta\nu)\big(S(L)\big)
 -(\eta\nu)\bigg(S(L)\setminus\bigcup_{\substack{L'\in\mathcal L_{k+1}\\ L'\subset L}}S(L')\bigg)
 \leq(1- c_0) (\eta\nu)\big(S(L)\big),
\]
where $c_0=[\eta]_{B_p(\nu)}^{-1}(1-N^{-1})^p$.
Iterating the stopping-time construction, we have
\begin{equation}\label{eq:generationmass}
 \sum_{L\in\mathcal L_k}(\eta\nu)\big(S(L)\big)
 \le(1-c_0)^k(\eta\nu)\big(S(I)\big).
\end{equation}
For $L\in\mathcal L_k$, then by the choice of $\mathcal L_{k+1}$ and \eqref{Dhat1}, there exists $C_1(\nu)>0$ such that
$$
 \frac{(\eta\nu)\big(S(L')\big)}{\nu\big(S(L')\big)}
 \le C_1\frac{(\eta\nu)\big(S(\widetilde{L'})\big)}{\nu\big(S(\widetilde{L'})\big)}\\
 \le C_1 N\frac{(\eta\nu)\big(S(L)\big)}{\nu\big(S(L)\big)}, \quad L'\in\mathcal L_{k+1},
$$where $L'\subset \widetilde{L'}$, $|\widetilde{L'}|=2|L'|$ and $\widetilde{L'}\in \mathcal{D}_{I}$.
Consequently, for every $L\in\mathcal L_k$,
\begin{equation}\label{eq:avgrowth}
 \frac{(\eta\nu)\big(S(L)\big)}{\nu\big(S(L)\big)}
 \le(C_1 N)^k\frac{(\eta\nu)\big(S(I)\big)}{\nu\big(S(I)\big)}.
\end{equation}
For $L\in\mathcal L_k$ and $z\in S(L)\setminus\bigcup\limits_{L'\in\mathcal L(L)}S(L')$, there exists $J\subset L$ such that $z\in T_K(J)$. Then \eqref{Ddef} implies
\[
\eta(z)\asymp \frac{1}{\nu\big(T_K(J)\big)}\int_{T_K(J)}\eta\nu dA
 \le \frac{(\eta\nu)\big(S(J)\big)}{\nu\big(T_K(J)\big)}
 \lesssim \frac{(\eta\nu)\big(S(J)\big)}{\nu\big(S(J)\big)}
 \le N\frac{(\eta\nu)\big(S(L)\big)}{\nu\big(S(L)\big)}.
\]
For $q>1$ and $L\in\mathcal L_k$, the above inequality together with \eqref{eq:avgrowth} gives
$$
\begin{aligned}
 \int_{S(L)\setminus\bigcup\limits_{L'\in\mathcal L(L)}S(L')}\eta^q \nu dA
 \lesssim&
 \left(N\frac{(\eta\nu)\big(S(L)\big)}{\nu\big(S(L)\big)}\right)^{q-1}
(\eta\nu)\bigg(S(L)\setminus\bigcup\limits_{L'\in\mathcal L(L)}S(L')\bigg)\\
 \le&
 N^{q-1}(C_1N)^{k(q-1)}
 \left(\frac{(\eta\nu)\big(S(I)\big)}{\nu\big(S(I)\big)}\right)^{q-1}
 (\eta\nu)\big(S(L)\big).
 \end{aligned}
$$
Summing first over $L\in\mathcal L_k$ and then over $k\ge0$, and applying \eqref{eq:generationmass}, we have
\begin{align*}
 \int_{S(I)}\eta^q \nu dA
 &=\sum_{k\ge0}\sum_{L\in\mathcal L_k}\int_{S(L)\setminus\bigcup\limits_{L'\in\mathcal L(L)}S(L')}\eta^q \nu dA\\
 &\lesssim
 \left(\frac{(\eta\nu)\big(S(I)\big)}{\nu\big(S(I)\big)}\right)^{q-1}
 \sum_{k\ge0}(C_1 N)^{k(q-1)}
 \sum_{L\in\mathcal L_k}(\eta\nu)\big(S(L)\big)\\
 &\lesssim
 \left(\frac{(\eta\nu)\big(S(I)\big)}{\nu\big(S(I)\big)}\right)^{q-1}
 (\eta\nu)\big(S(I)\big)
 \sum_{k\ge0}\bigl((1-c_0)(C_1 N)^{q-1}\bigr)^k.
\end{align*}
Since $0<c_0<1$, we may choose $q>1$ sufficiently close to $1$ so that
$(1-c_0)(C_1 N)^{q-1}<1.$
Therefore
$$
 \frac1{\nu\big(S(I)\big)}\int_{S(I)}\eta^q \nu dA
 \lesssim \left(\frac{(\eta\nu)\big(S(I)\big)}{\nu\big(S(I)\big)}\right)^q,
$$ which implies
$\eta\in RH_{q,\nu}$. This completes the proof.
\end{proof}
We are now ready to prove Theorem \ref{main2}.
\vspace{8pt}

\noindent \textbf{Proof of Theorem \ref{main2}.}
It follows from Lemma \ref{normeq} that $T_f^\nu$ is compact on $A_\omega^p$ if and only if it is compact on $A_{\sigma_{\omega,\nu ,r}}^p$ for each fixed $r\geq R$. Then by Corollary \ref{BpRH}, it suffices to verify that there exists $q>1$ such that $\sigma_{\omega,\nu ,r}\nu^{-1}\in B_p(\nu)\cap RH_{q,\nu}$ and $(\sigma_{\omega,\nu ,r}\nu^{-1})^{-\frac{p'}{p}}\in RH_{q,\nu}$.

We first show that if $\omega\nu^{-1}\in B_p(\nu)$, then $\sigma_{\omega,\nu ,r}\nu^{-1}\in B_p(\nu)$.
Indeed, it follows from Lemma~\ref{normeq} that $\sigma_{\omega,\nu ,r}(S)\lesssim \omega(S)$ for all Carleson square $S$.
Moreover, we know that $\vartheta\in \mathcal{D}(\mathbb{D})$ since $\vartheta\nu^{-1}\in B_{p'}(\nu)$. Then Jensen's inequality together with Fubini's theorem and \eqref{DhatD} yields
$$
\begin{aligned}
\int_{S(a)} \bigg(\frac{\sigma_{\omega,\nu ,r}(\zeta)}{\nu(\zeta)}\bigg)^{-\frac{p'}{p}}\nu(\zeta) dA(\zeta)
\leq &\int_{S(a)} \frac{1}{\nu\big(D(\zeta,r)\big)}\left(\int_{D(\zeta,r)}\big(\omega\nu^{-1}\big)^{-\frac{p'}{p}}\nu dA\right)\nu(\zeta) dA(\zeta)\\
=& \int_{\{z\in\mathbb{D}:S(a)\cap D(z,r)\neq\o \}}\left( \int_{D(z,r)\cap S(a)}\frac{\nu(\zeta)dA(\zeta)}{\nu\big(D(\zeta,r)\big)}\right)
\vartheta(z)dA(z)\\
\lesssim&\vartheta (S(a)),\quad |a|>\tanh r.
\end{aligned}
$$
Then we have $\sigma_{\omega,\nu ,r}\nu^{-1}\in B_p(\nu)$, and consequently $(\sigma_{\omega,\nu ,r}\nu^{-1})^{-\frac{p'}{p}}\in B_{p'}(\nu)$.

Notice that $\sigma_{\omega,\nu ,r}\nu^{-1}$ is essentially constant on hyperbolically bounded regions whenever $r\geq R$. By Lemma \ref{invarRH}, it follows that $\sigma_{\omega,\nu ,r}\nu^{-1}\in RH_{q_1,\nu}$ for some $q_1>1$. Similarly, $(\sigma_{\omega,\nu ,r}\nu^{-1})^{-\frac{p'}{p}}\in RH_{q_2,\nu}$ for some $q_2>1$. Taking $q=\min\{q_1, q_2\}$ completes the proof.
\qed
\color{black}

\end{document}